\documentclass[12pt]{amsart}

\usepackage{amsmath,amssymb,amsfonts,amscd}
\usepackage{url}
\usepackage{color}
\usepackage[usenames,dvipsnames]{xcolor}
\usepackage[all]{xy}
\usepackage{hyperref}
\usepackage{mathtools}
\usepackage{mathrsfs}

\DeclareMathOperator{\Res}{Res}
\DeclareMathOperator{\Ext}{Ext}
\DeclareMathOperator{\Hom}{Hom}
\DeclareMathOperator{\Gr}{Gr}
\DeclareMathOperator{\rk}{rk}
\DeclareMathOperator{\id}{Id}
\DeclareMathOperator{\sExt}{\sE\!\mathit{xt}}

\DeclareMathOperator{\usExt}{\underline{\sE\!\mathit{xt}}}
\DeclareMathOperator{\usHom}{\underline{\sH\!\mathit{om}}}
\renewcommand{\Im}{\operatorname{Im}}

\begin{document} 
\theoremstyle{plain}
\newtheorem{thm}{Theorem}[section]
\newtheorem{lem}[thm]{Lemma}
\newtheorem{cor}[thm]{Corollary}
\newtheorem{prop}[thm]{Proposition}
\newtheorem{prop-defn}[thm]{Proposition-Definition}
\newtheorem{question}[thm]{Question}
\newtheorem{claim}[thm]{Claim}

\theoremstyle{definition}
\newtheorem{remark}[thm]{Remark}
\newtheorem{defn}[thm]{Definition}
\newtheorem{ex}[thm]{Example}
\newtheorem{conj}[thm]{Conjecture}
\numberwithin{equation}{section}
\newcommand{\eq}[2]{\begin{equation}\label{#1}#2 \end{equation}}
\newcommand{\ml}[2]{\begin{multline}\label{#1}#2 \end{multline}}
\newcommand{\ga}[2]{\begin{gather}\label{#1}#2 \end{gather}}
\newcommand{\mc}{\mathcal}
\newcommand{\mb}{\mathbb}
\newcommand{\surj}{\twoheadrightarrow}
\newcommand{\inj}{\hookrightarrow}
\newcommand{\red}{{\rm red}}
\newcommand{\codim}{{\rm codim}}
\newcommand{\rank}{{\rm rank}}
\newcommand{\Pic}{{\rm Pic}}
\newcommand{\Div}{{\rm Div}}
\newcommand{\im}{{\rm im}}
\newcommand{\Spec}{{\rm Spec \,}}
\newcommand{\Sing}{{\rm Sing}}
\newcommand{\Char}{{\rm char}}
\newcommand{\Tr}{{\rm Tr}}
\newcommand{\Gal}{{\rm Gal}}
\newcommand{\Min}{{\rm Min \ }}
\newcommand{\Max}{{\rm Max \ }}
\newcommand{\ti}{\times }
% Skriptbuchstaben
\newcommand{\sA}{{\mathcal A}}
\newcommand{\sB}{{\mathcal B}}
\newcommand{\sC}{{\mathcal C}}
\newcommand{\sD}{{\mathcal D}}
\newcommand{\sE}{{\mathcal E}}
\newcommand{\sF}{{\mathcal F}}
\newcommand{\sG}{{\mathcal G}}
\newcommand{\sH}{{\mathcal H}}
\newcommand{\sI}{{\mathcal I}}
\newcommand{\sJ}{{\mathcal J}}
\newcommand{\sK}{{\mathcal K}}
\newcommand{\sL}{{\mathcal L}}
\newcommand{\sM}{{\mathcal M}}
\newcommand{\sN}{{\mathcal N}}
\newcommand{\sO}{{\mathcal O}}
\newcommand{\sP}{{\mathcal P}}
\newcommand{\sQ}{{\mathcal Q}}
\newcommand{\sR}{{\mathcal R}}
\newcommand{\sS}{{\mathcal S}}
\newcommand{\sT}{{\mathcal T}}
\newcommand{\sU}{{\mathcal U}}
\newcommand{\sV}{{\mathcal V}}
\newcommand{\sW}{{\mathcal W}}
\newcommand{\sX}{{\mathcal X}}
\newcommand{\sY}{{\mathcal Y}}
\newcommand{\sZ}{{\mathcal Z}}
% Sonderbuchstaben mit Doppellinie
\newcommand{\A}{{\mathbb A}}
\newcommand{\B}{{\mathbb B}}
\newcommand{\C}{{\mathbb C}}
\newcommand{\D}{{\mathbb D}}
\newcommand{\E}{{\mathbb E}}
\newcommand{\F}{{\mathbb F}}
\newcommand{\G}{{\mathbb G}}
\renewcommand{\H}{{\mathbb H}}
\newcommand{\I}{{\mathbb I}}
\newcommand{\J}{{\mathbb J}}
\newcommand{\M}{{\mathbb M}}
\newcommand{\N}{{\mathbb N}}
\renewcommand{\P}{{\mathbb P}}
\newcommand{\Q}{{\mathbb Q}}
\newcommand{\R}{{\mathbb R}}
\newcommand{\T}{{\mathbb T}}
\newcommand{\U}{{\mathbb U}}
\newcommand{\V}{{\mathbb V}}
\newcommand{\W}{{\mathbb W}}
\newcommand{\X}{{\mathbb X}}
\newcommand{\Y}{{\mathbb Y}}
\newcommand{\Z}{{\mathbb Z}}
\newcommand{\pic}{{\text{Pic}(C,\sD)[E,\nabla]}}
\newcommand{\ocd}{{\Omega^1_C\{\sD\}}}
\newcommand{\oc}{{\Omega^1_C}}
\newcommand{\al}{{\alpha}}
\newcommand{\be}{{\beta}}
\newcommand{\ta}{{\theta}}
\newcommand{\ve}{{\varepsilon}}
\newcommand{\lr}[2]{\langle #1,#2 \rangle}
\newcommand{\nnn}{\newline\newline\noindent}
\newcommand{\nn}{\newline\noindent}
\newcommand{\snote}[1]{{\color{blue} Spencer: #1}}

\newcommand{\onote}[1]{{\color{blue} Omid: #1}}
\newcommand{\jnote}[1]{{\color{orange} Javier: #1}}
\newcommand{\jinote}[1]{{\color{teal} Jose: #1}}
\newcommand{\vs}{\mathcal V} 
\newcommand{\ws}{\mathcal W} 
\newcommand{\spanningf}{\mathcal{SF}}
\newcommand{\gr}{\mathrm{gr}}

%%%%%%%%%%%%%%%%%%%%%%
\newcommand{\ok}{\mathbb K} 
\newcommand{\ovar}{Y} 
\newcommand{\Diag}{\mathrm{Diag}}
\newcommand{\green}{\mathfrak g}
\newcommand{\p}{\mathbf{p}}
\newcommand{\ps}{\mathbf{p}}
\newcommand{\oA}{\mathfrak A}
\newcommand{\oB}{\mathfrak B}
\newcommand{\oD}{\mathfrak D}

\newcommand{\ad}{\operatorname{ad}}
\newcommand{\K}{\mathbb K}
\newcommand{\an}{\operatorname{an}}
\newcommand{\val}{\mathrm{val}}
\newcommand{\Ar}{\mathrm{Ar}}
\newcommand{\onto}{\twoheadrightarrow}  
\newcommand{\oL}{\mathrm{L}}
\newcommand{\rl}{\mathit{l}}

\title{Feynman Amplitudes and Limits of Heights}
\author[O. Amini]{Omid Amini}
\address{CNRS--DMA, \'Ecole Normale Sup\'erieure, 45 rue d'Ulm, Paris}
\email{oamini@math.ens.fr}

\author[S. Bloch]{Spencer Bloch}
\address{5765 S. Blackstone Ave., Chicago, IL 60637, USA}
\email{spencer\_bloch@yahoo.com}

\author[J. I. Burgos Gil]{Jos\'{e} Ignacio Burgos GIl}
\address{Instituto de Ciencias Matem\'aticas (CSIC-UAM-UCM-UCM3).
  Calle Nicol\'as Ca\-bre\-ra~15, Campus UAM, Cantoblanco, 28049 Madrid,
  Spain.} 
\email{burgos@icmat.es}

\author[J. Fres\'{a}n]{Javier Fres\'{a}n}
\address{ETH, D-MATH, R\"amistrasse 101, CH-8092 Z\"urich, Switzerland}
\email{javier.fresan@math.ethz.ch}

\begin{abstract} 
We investigate from a mathematical perspective how Feynman amplitudes appear in the low-energy limit of string amplitudes. In this paper, we prove the convergence of the integrands. 
We derive this from results describing the asymptotic behavior of the height pairing between degree-zero divisors, as a family of curves degenerates. These are obtained by means of the nilpotent orbit theorem in Hodge theory.
\end{abstract}
\maketitle

\bigskip

\begin{flushright}
\textit{\`A Jean-Pierre Serre, en t\'emoignagne d'admiration}
\end{flushright}

\section{Introduction}

This paper grew out of an attempt to understand from a mathematical perspective the idea we learned from physicists that Feynman amplitudes should arise in the low-energy limit $\alpha' \to 0$ of string theory amplitudes, cf. \cite{T} and the references therein. Throughout we work in space-time $\R^D$ with a given Minkowski bilinear form $\langle\cdot\,,\cdot\rangle$. 

\smallskip

String amplitudes are integrals over the moduli
space $\mathcal M_{g,n}$ of genus $g \geq 1$ curves with $n$ marked
points. They are associated to a fixed collection of external momenta
$\underline \p=(\ps_1, \dots, \ps_n)$, which are vectors in $\mathbb
R^D$ satisfying the conservation law $\sum_{i=1}^n \ps_i =0$. Up to
some factors carrying information about the physical process being
studied, the string amplitude can be written as (see
e.g.~\cite[p.182]{Vor}) 
\begin{equation}\label{string-int}
A_{\alpha'}(g, \underline \p)= \int_{\mathcal M_{g,n}} \exp(-i\, \alpha'  \mathcal{F}) \,\,d\nu_{g,n}. 
\end{equation} In this expression, $d\nu_{g,n}$ is a volume form on $\mathcal{M}_{g, n}$, independent of the momenta, $\alpha'$ is a positive real number, which one thinks of as the square of the string length, and $\mathcal{F} \colon \mathcal M_{g,n} \rightarrow \mathbb R$ is the continuous function defined at the point $[C, \sigma_1, \ldots, \sigma_n]$ of $\mathcal M_{g, n}$ by 
\begin{equation*}
\mathcal{F}([C, \sigma_1, \ldots, \sigma_n])=\sum_{1 \leq i, j \leq n} \langle \ps_i, \ps_j \rangle \, \green'_{\Ar,C}(\sigma_i, \sigma_j), 
\end{equation*} where $\green'_{\Ar, C}$ denotes a regularized version of the canonical Green function on $C$ so that it takes finite values on the diagonal. %(see section \ref{sec:mainresult} for the precise definition). 

\smallskip

On the other hand, correlation functions in quantum field theory are calculated using Feynman amplitudes, which are certain finite dimensional integrals associated to graphs. Recall that a (massless) Feynman graph $(G, \p)$ consists of a finite graph $G=(V,E)$, with vertex and edge sets $V$ and $E$, respectively, together with a collection of external momenta $\underline \p=(\ps_v)_{v \in V}$, $\ps_v \in \R^D$, such that $\sum_{v\in V} \ps_v=0$. To the Feynman graph $(G, \underline \p)$ one associates two polynomials in the variables $\underline{Y}=(Y_e)_{e \in E}$. The first Symanzik $\psi_G$, which depends only on the graph $G$, is given by the following sum over the spanning trees of $G$:
$$
\psi_G(\underline{Y})=\sum_{T \subseteq G} \prod_{e \notin T} Y_e.
$$ The second Symanzik polynomial $\phi_G$, depending on the external momenta as well, admits the expression 
$$
\phi_G(\underline \p, \underline{Y})=\sum_{F \subseteq G} q(F) \prod_{e \notin F} Y_e.
$$ Here $F$ runs through the spanning $2$-forests of $G$, and $q(F)$ is the real number $-\langle \p_{F_1}, \p_{F_2}\rangle$, where $\p_{F_1}$ and $\p_{F_2}$ denote the total momentum entering  the two connected  components $F_1$ and $F_2$ of ${F}$. The polynomial $\phi_G$ is quadratic in $\underline \p$ and it will be also convenient to consider the corresponding bilinear form, which we denote by $\phi_G(\underline \p, \underline \p', \underline{Y})$.  

\smallskip

One of the various representations of the Feynman amplitude associated to $(G, \underline \p)$ is, up to some elementary factors which we omit, 
\begin{equation}\label{feynman-int}
I_G(\underline \p) =\int_{[0,\infty]^E}\exp(-i\,\phi_G/\psi_G)\,\, d\pi_G,
\end{equation} where $d\pi_G$ denotes the volume form $\psi_G^{-D/2} \prod_E dY_e$ on $[0,\infty]^E$. This can be found e.g. in formula (6-89) of \cite{IZ}.
If one interprets the locus of integration as the space of metrics (i.e. lengths of edges) on $G$, then \eqref{feynman-int} looks like a path integral with the action $\phi_G/\psi_G$.

\smallskip

Although both amplitudes diverge in general, one may still ask if
$I_G(\underline \p)$ is related to the asymptotic of
$A_{\alpha'}(g, \underline \p)$ when $\alpha'$ goes to zero, as
physics suggests.  The graph $G$ appears as the dual graph of a stable
curve $C_0$ with $n$ marked points lying on the boundary of the
Deligne-Mumford compactification $\overline{\mathcal M}_{g, n}$
(recall that the irreducible components $X_v$ of $C_0$ are indexed by
the vertices of $G$, whereas the singular points correspond to the
edges).  The question can be then split into two different problems,
namely (i) the convergence of the integrands, and (ii) the convergence
of the measure $\nu_{g,n}$, in an appropriate sense, and along the
boundary of $\overline{\mathcal M}_{g, n}$, to a linear combination of
the measures $\pi_G$ for the (marked) dual graphs $G$ associated to
the~strata. Our main result in this paper answers question (i) in the affirmative
when the external momenta satisfy the ``on shell'' condition: the integrand in string theory converges indeed
to the integrand appearing in the Feynman amplitude. 

\smallskip

To make this more
precise, let us consider the versal analytic deformation $\pi \colon
\mathcal{C}' \to S'$ of the marked curve
$C_0$, which we think of as a smooth neighborhood of $C_0$ in the
analytic stack $\overline{\mathcal{M}}_{g,n}$. Here $S'$ is a polydisc of
dimension $3g-3+n$, the total space $\mathcal{C}'$ is regular and
$\mathcal{C}_0'$, the fibre of $\pi$ at $0$, is isomorphic to
$C_0$. For each edge $e \in E$, let $D_e \subset S'$ denote the divisor
parametrizing those deformations in which the point associated to $e$
remains singular. Then $D=\bigcup_{e \in E} D_e$ is a normal crossings
divisor whose complement $U'=S' \setminus D$ can be identified with
$(\Delta^\ast)^E \times \Delta^{3g-3-|E|+n}$. Over $U'$, the fibres
$\mathcal{C}'_s$ are smooth curves of genus $g$. Moreover, the versal family comes together with $n$ disjoint sections $\sigma
_{i}\colon S'\to \sC'$ which do not meet the double points of
$C_{0}$. We denote  by $\underline \p^G = 
(\ps_v^G)_{v\in V}$ the \emph{restriction} of $\underline \p$ to
$G$. By this we mean that, for each vertex $v \in V$, the external
momentum $\p_v^G$ is obtained by summing those $\p_i$ associated to
the sections $\sigma_i$ which meet $C_0$ on the irreducible
component $X_v$. 

\smallskip

An \textit{admissible segment} is a continuous maps $\underline t
\colon [0, \varepsilon] \to S'$ from an interval of length
$\varepsilon>0$ such that $\underline t((0, \varepsilon]) \in U'$ and, letting $t_e$ denote the coordinate
corresponding to $e\in E$ in the factor $(\Delta^\ast)^E$ of $U'$, the limit
$\lim_{\alpha' \to 0} |t_e(\alpha')|^{\alpha'}$ exists and belongs to
$(0, 1)$. To any admissible segment we attach a collection $\underline{Y}=(Y_e)_{e\in E}$ of positive real
numbers (the edge lengths) as follows: 
$$
Y_e=-\lim_{\alpha' \to 0} \log |t_e(\alpha')|^{\alpha'}. 
$$

\begin{thm}[cf. Theorem \ref{thm:5}]\label{thm:mainintro} Let $C_0$ be
  a stable curve of genus $g \geq 1$ with $n$ marked points $\sigma_1, \dots,
  \sigma_n$ and dual graph $G=(V,E)$, and let $\underline \p =
  (\ps_1,\dots,\ps_n)$ be a collection of external momenta 
  satisfying the conservation law $\sum_{i=1}^n \p_{i}=0$ and the ``on shell''
  condition $\langle \p_{i}, \p_{i} \rangle=0$ for all $i$. Then, for any admissible
  segment $\underline t \colon I \to \overline{\mathcal{M}}_{g, n}$ such
  that $\underline t(0)=[C_0, \sigma_1, \dots, \sigma_n]$, we have 
  $$
  \lim_{\alpha' \to 0} \alpha'
  \sF(\underline t(\alpha'))=\frac{\phi_G(\underline \p^G,
    \underline{Y})}{\psi_G(\underline{Y})}, 
  $$
  where $\underline Y=(Y_e)_{e \in E}$ denotes the edge lengths
  determined by $\underline t$.
\end{thm}

\smallskip

We derive the above theorem from results describing the asymptotic
behaviour of the archimedean height pairing. We start with the case of
disjoint divisors. We consider
the versal analytic deformation $\pi \colon \mathcal{C} \to S$ of
$C_0$ (without the marked points), which we think of as a smooth
neighborhood of $C_0$ in the 
analytic stack $\overline{\mathcal{M}}_{g}$. Now $S$ is a polydisc of
dimension $3g-3$. Again the total space $\mathcal{C}$ is regular and
we repeat the construction above letting $D_e \subset S$ denote the divisor
parametrizing those deformations in which the point associated to $e$
remains singular. Then $D=\bigcup_{e \in E} D_e$ is a normal crossings
divisor whose complement $U=S \setminus D$ can be identified with
$(\Delta^\ast)^E \times \Delta^{3g-3-|E|}$. To accommodate
external momenta, we assume that we are given two collections of sections of $\pi$, which we denote by 
$\sigma_1=(\sigma_{\ell, 1})_{\ell=1, \ldots, n}$ and $\sigma_2=(\sigma_{\ell, 2})_{\ell=1, \ldots, n}$. Since $\mathcal{C}$ is regular, the
points $\sigma_{\rl,i}(0)$ lie on the smooth locus of $C_0$.  We label
the markings with two vectors
$\underline \p_1=(\ps_{\rl,1})_{\rl=1}^n$ and
$\underline \p_2=(\ps_{\rl,2})_{\rl=1}^n$ with $\ps_{\rl,i} \in \R^D$
subject to the conservation of
momentum, %$\sum_{\rl=1}^n \ps_{\rl, i}=0$,
thus obtaining a pair of relative degree zero $\R^D$-valued divisors
$$
\mathfrak A_s = \sum_{\rl=1}^n \p_{\rl, 1} \sigma_{\rl, 1}, \qquad \mathfrak
B_s = \sum_{\rl=1}^n \p_{\rl,2} \sigma_{\rl,2}.
$$

\smallskip

We first assume that $\sigma_1$ and $\sigma_2$
are \textit{disjoint} on each fiber of $\pi$. 
Recall that to any pair $\mathfrak A$, $\mathfrak B$ of degree zero
(integer-valued) divisors with disjoint support on a smooth projective
complex curve $C$, one associates a real number, the
\textit{archimedean height}
\begin{equation*}
\langle \mathfrak A, \mathfrak B \rangle=\mathrm{Re}(\int_{\gamma_{\mathfrak{B}}} \omega_{\mathfrak A}), 
\end{equation*} by integrating a canonical logarithmic differential
$\omega_{\mathfrak{A}}$ with residue $\mathfrak{A}$ along any
$1$-chain $\gamma_{\mathfrak{B}}$ supported on $C \setminus
|\mathfrak{A}|$ and having boundary $\mathfrak B$. Coupling with the
Minkowski bilinear form on $\R^D$, the definition extends to
$\R^D$-valued divisors. We thus get a real-valued function 
\begin{equation}\label{eq:3}
s \mapsto \langle \mathfrak A_s, \mathfrak B_s \rangle. 
\end{equation}  

\smallskip

For each $e\in E$, denote by $s_e$ the coordinate in the factor corresponding to $e$ in $U=(\Delta^\ast)^E \times \Delta^{3g-3-|E|}$, write $y_{e}=\frac{-1}{2\pi }\log|s_{e}|$ and put
$\underline y=(y_{e})_{e\in E}$. After shrinking $U$ if necessary, 
the asymptotic of the height pairing is given by the following result:

\begin{thm}[cf. Corollary \ref{cor:1}] \label{thm:6} Assume, as above, that $\sigma_1$ and $\sigma_2$ are disjoint on each fibre. Then there exists a bounded function $h\colon U\
  \to \R$ such that 
  \begin{equation}\label{eq:22}
    \langle \oA_{s},\oB_{s}\rangle=
    2\pi \frac{\phi_G(\underline \p_1^G, 
  \underline \p_2^G,\underline y)}{\psi_G(\underline y)}
+h(s). 
  \end{equation}
\end{thm}

Theorem \ref{thm:6} only deals with disjoint sections. In order to
derive Theorem \ref{thm:mainintro} we need to allow the supports of the divisors to intersect, which requires a regularization of the height pairing. Pointwise, the regularization depends on the choice of a metric on the tangent space of the given curve. To regularize
the height pairing globally we choose a smooth $(1,1)$-form $\mu$  on
$\pi ^{-1}(U)$ such that the restriction to each fibre $\mathcal{C}_{s}$ is positive. Using $\mu$ we define a regularized height pairing \begin{math}
  \langle \oA_{s},\oA_{s}\rangle^{'}_{\mu }
\end{math} (see Section \ref{sec:conv-integr}). In the case where $g \geq 1$ and $\mu$ is the Arakelov metric $\mu _{\Ar}$ we recover the
function $\sF$ from the string amplitude: 
\begin{displaymath}
  \sF([\sC_{s},\sigma _{1}(s),\dots,\sigma _{n}(s)])=\langle \oA_{s},\oA_{s}\rangle^{'}_{\mu _{\Ar}}.
\end{displaymath}

\smallskip

The asymptotic of the regularized height pairing is described in the
following result: 

\begin{thm}[cf. Theorem \ref{thm:8} and Corollary \ref{cor55}]\label{thm:7} Let $\underline \p=(\ps_i)_{i=1, \ldots, n}$ be external momenta satisfying the conservation law and $(\sigma_i)_{i=1, \ldots, n}$ a collection of sections $\sigma_i \colon S \to \mathcal{C}$. Put $\oA_{s}=\sum
  \p_{i}\sigma _{i}(s)$, and let $\mu $ be a smooth $(1,1)$-form on
  $\pi^{-1}(U)$ whose restriction to each curve $\mathcal{C}_{s}$ is
  positive. Assume that one of the following conditions hold: 
  \begin{enumerate}
  \item $\mu $ extends to a continuous $(1,1)$-form on $\sC$, or  
  \item the $\ps_i$ satisfy the ``on shell'' condition
    $\langle \p_{i}, \p_{i} \rangle=0$. 
  \end{enumerate}

  Then there exists a bounded function $h\colon U\to \R$ such that
  \begin{displaymath}
    \langle \oA_{s},\oA_{s}\rangle^{'}_{\mu}=2\pi
    \frac{\phi_G(\underline \p^G,\underline y)}{\psi_G(\underline y)}
    +h(s).  
  \end{displaymath}
\end{thm}

\smallskip

To get Theorem \ref{thm:mainintro} from Theorem~\ref{thm:7} we first
observe that the latter can be easily extended to the versal family
$\sC'\to S'$ of
genus $g$ curves with $n$ marked points. For any admissible segment
$\underline t\colon I\to U'$, we have 
\begin{align*}\label{1.3formula}
 \lim_{\alpha' \to 0} \alpha' \sF(\underline t(\alpha'))
  &= 
    \lim_{\alpha' \to 0}\left[\frac{\phi_G(\underline \p^G,(-\log
    |t_e(\alpha')|^{\alpha'})_{e})}{\psi_G((-\log
    |t_e(\alpha')|^{\alpha'})_{e})}+\alpha 'h(s)\right]\\
  &=\frac{\phi_G(\underline \p^G,\underline Y)}{\psi_G(\underline Y)},  
\end{align*}
where we have used that $\phi_G \slash \psi_G$ is homogeneous of degree one, as well as the boundedness of $h$.
\smallskip

\begin{remark} If one wants to compute the quantum field theory amplitude \eqref{feynman-int} for ``off shell'' momenta as a limit of heights in the spirit of this paper,  the surprising ``on shell'' condition in Theorem \ref{thm:mainintro} can be avoided by simply taking momenta $\underline \p^G=\underline \p_1^G=\underline \p_2^G$ and disjoint multisections $\sigma_1, \sigma_2$ which have the same intersection data with components of the curve at infinity. One can combine equation \eqref{eq:22} and the above limit calculation, noting that $\phi_G(\underline \p^G,\underline Y) = \phi_G(\underline \p_1^G,\underline \p_2^G,\underline Y)$. 
\end{remark}

The proofs of Theorems \ref{thm:6} and \ref{thm:7} are based on the Hodge theoretic
interpretation of the archimedean height. Since both sides of the
equality \eqref{eq:22} are bilinear in the
momenta, we can reduce to the case of
integer-valued divisors. Then 
$\mathfrak{A}_s$ and $\mathfrak{B}_s$ define a \textit{biextension}
mixed Hodge structure $H_{\mathfrak{B_s}, \mathfrak{A_s}}$ with graded
pieces $\Z(1), H^1(\mathcal{C}_s, \Z(1)), \Z(0)$. The moduli space of
such biextensions is the $\mathbb{C}^\times$-bundle associated to the
Poincar\'e line bundle over
$J(\mathcal{C}_s) \times \widehat{J(\mathcal{C}_s)}$, and one recovers
the archimedean height by evaluating its canonical metric at
$H_{\mathfrak{B}_s, \mathfrak{A}_s}$. As $s$ varies, the biextensions
$H_{\mathfrak{B}_s, \mathfrak{A}_s}$ fit together into an admissible variation of
mixed Hodge structures over $U$. We shall write the
period map
$$
\widetilde{\Phi} \colon \widetilde{U} \longrightarrow \H_{g}\times
\text{Row}_{g}(\C)\times \text{Col}_{g}(\C)\times \C,
$$ 
where $\widetilde{U}$ is the universal cover of $U$ and
$\mathbb{H}_{g}$ the Siegel upper half-space. If $\mathcal{P}^\times$
denotes the Poincar\'e bundle over the universal family of abelian
varieties and their duals, $\widetilde{\Phi}$ descends to the map
$\Phi \colon U\to \mathcal{P}^\times$ which sends $s$ to
$H_{\mathfrak{B}_s, \mathfrak{A}_s}$. Then (a weak version of) the
nilpotent orbit theorem 
for variations of mixed Hodge structures allows us to describe the
asymptotic of the height pairing. 
%concluding the proof of Theorems~\ref{main:intro} and~\ref{main:intro2}.
\smallskip

In order to prove Theorem~\ref{thm:7}, we need to study the effect of
the regularization process. To this end, we consider an analytic
family of sections $\sigma_{j}^u$ 
parametrized by $u$ in a small disc such that $\sigma_{j}^0 =\sigma_{j}$
and, for each $u\not =0$, the sections
$\sigma_{i}^{u}$ and $\sigma_{j}$ are 
disjoint for all $i, j$. We also choose a suitable rationally equivalent divisor $\oA+\mathrm{div}(f)$. When the
metric extends continuously to $\sC$, the difference between the
regularizations obtained using the metric $\mu$ and changing the section through the function $f$ is bounded, thus proving
the result. On the other hand, if the external momenta satisfy the ``on shell'' condition, then all divergent
terms vanish, which makes the regularization process independent of the
metric. In this way we deduce the result in the ``on shell'' case for
non-continuous metrics from the result for continuous metrics.  

\smallskip

We would like to end this introduction by mentioning that the
asymptotic of the height pairing has been previously studied from a mathematical perspective 
in~\cite{Hain, HP, HJ, Lear, pearl:sl2}. In particular, an analogue of
theorems~\ref{thm:6} and \ref{thm:7} for families of curves
over a one-dimensional base was established by Holmes and de Jong
in~\cite{HJ}. Some Hodge theoretic aspects of stable curves appear
already in the work of Hoffman~\cite{Hoffman}.  We are less familiar with the physics literature, but many of the ideas in this paper are discussed from a physics viewpoint in \cite{Vanhove}.

\smallskip

The paper is organized as follows. In Section \ref{sec:sym}, we
discuss Symanzik polynomials in an abstract setting, convenient 
for the sequel. Section~\ref{riemannsurf} is devoted to the study of the local monodromy of the analytic versal deformation. In Section \ref{sec:arch-heights-poinc}, we recall
the definition of the archimedean height pairing, as well as its
interpretation in terms of biextensions and the Poincar\'e
bundle. Section \ref{sec:nilp} contains the proof of Theorem
\ref{thm:6}, for which we need to write the
period map and use part of the nilpotent orbit theorem. Finally, in Section
\ref{sec:conv-integr} we put everything together to get the convergence
of the integrands.  

\medskip

{\bf Acknowledgments.} This project grew out of discussions
started during the Clay Mathematics Institute Summer School
\emph{Periods and motives: Feynman amplitudes in the 21st century},
held at the ICMAT in July 2014, and followed by a visit to the \'{E}NS
in January 2015. We would like to thank these institutions for their
hospitality. We are very grateful to the anonymous referee for a long
list of comments which greatly improved the
paper. O. A. thanks Amir-Kian Kashani-Poor for helpful
discussions. S. B. would like to acknowledge helpful conversations
with Piotr Tourkine and Pierre Vanhove. Most particularly, he would
like to acknowledge correspondence and conversations going back ten
years or more with Prof. Kazuya Kato; this paper may be viewed as a
kind of first step to understand his work~\cite{KU}. J. I. B. would like to thank
David Holmes and Robin de
Jong for many helpful discussions. He was
partially supported by the MINECO research projects MTM2013-42135-P
and ICMAT Severo Ochoa project SEV-2015-0554, and 
the DFG project SFB 1085 ``Higher Invariants''. Finally, J. F. would
like to thank
Rahul Pandharipande for stimulating questions when a first version of
these results was presented at his \textit{Algebraic geometry and
  moduli} seminar. He also acknowledges support from the SNFS grants
200021-150099 and 200020-162928.

\tableofcontents

\section{Symanzik Polynomials}\label{sec:sym}

In this section, we define the first and the second Symanzik polynomials in an abstract setting. We then show how to recover the usual formulas for $\psi_G$ and $\phi_G$ in the case of graphs~\cite{ELOP}. 
Throughout, if $\ok$ is a field and $E$ a finite set, we write $\ok^E=\bigl\{\,\sum_{e\in E} \kappa_ee\ |\ \kappa_e \in \ok\,\bigr\}$. 
For each $e\in E$, we denote by $e^\vee: \ok^E \to \ok$ the functional which takes the $e$-th 
coordinate of a vector.

\subsection{Abstract setting} Let $H$ be a vector space of finite dimension $h$ over a field $\ok$, 
and suppose we are given a finite set $E$ of cardinality at least $h+1$, and an embedding $\iota: H \inj \ok^E$. Abusing notation, we write $e^\vee$ as well for the composition $H \inj \ok^E \to \ok$. 
The function which sends $x \in H$ to $e^\vee(x)^2$ defines a rank one quadratic form $e^{\vee,2}$ on $H$. When needed, we denote by 
$\langle\cdot\,,\cdot\rangle_e$ the corresponding bilinear form. 

\smallskip

If we fix a basis $\gamma_1, \dots, \gamma_h$ of $H$, we can identify
the quadratic form $e^{\vee,2}$ with an $h \times h$ symmetric matrix
$M_e$ of rank one  so that, thinking of elements of $H$ as column
vectors, 
we have
\begin{equation}\label{eq:1}
 e^{\vee,2}(x) = \prescript{\mathrm t}{}{x}M_e x.
\end{equation}

\smallskip

Let $\underline{Y}=\{\ovar_e\}_{e\in E}$ be a collection of variables
indexed by $E$, and consider the matrix $M=\sum_{e\in E}\ovar_eM_e$
(associated to the quadratic form $\sum_{e\in E} \ovar_e e^{\vee,2}$).
It is an $h\times h$ symmetric matrix whose entries are linear forms
in the $\ovar_e$. The linear map $M: H \to H^\vee$ is in fact
canonical and independent of the choice of a basis of $H$.
 
\begin{defn}\label{def:1st} The first Symanzik polynomial $\psi(H,\underline \ovar)$ associated to the configuration $H\hookrightarrow \ok^E$ 
is defined as
\begin{equation*}
\psi(H,\underline \ovar)=\det(M).
\end{equation*}
\end{defn}

\begin{remark}\label{rem:ind} Note that this definition depends on the
  choice of a basis of $H$. For a different basis, $M$ is replaced by
  $\prescript{\mathrm t}{}{P}M P$, where $P$ is the $h \times h$
  invertible matrix transforming one basis into the other, so the
  determinant gets multiplied by an element of $\ok^{\times, 2}$. The
  same argument shows that when $H=L \otimes_{\Z} \ok$ for a
  sublattice $L$ of $\Z^E$ and we restrict to bases coming from $L$,
  the first Symanzik polynomial is well-defined.  
\end{remark}

Let $\ws=\ok^E/H$. For any nonzero $w\in \ws$, we define
$H_w \subseteq \ok^E$ as the $(h+1)$-dimensional subspace of vectors
in $\ok^E$ whose images in $\ws$ lie in the line 
spanned by $w$.  Choosing a vector $\omega \in \ok^E$ in the preimage
of $w$, we can extend the basis $\{\gamma_1,\dots, \gamma_h\}$ of $H$
to a basis $\{\gamma_1, \dots, \gamma_h,\omega\}$ of $H_w$. The first
Symanzik polynomial $\psi(H_w,\underline \ovar)$ with respect to this
basis of $H_w$ yields the second Symanzik.
 
\begin{lem} $\psi(H_w,\underline \ovar)$ does not depend on the choice of $\omega$.
\end{lem} 
 
\begin{proof}
Let $\langle\,.\,,.\, \rangle = \sum_{e\in E} \ovar_e\langle\,.\,,.\,\rangle_e$ be the bilinear form on $H_w$ associated to the quadratic form $\sum \ovar_e e^{\vee,2}$. The polynomial 
$\psi(H_w, \underline \ovar)$  is the determinant of $\langle\,.\,,.\, \rangle$ with respect to the basis $\{\gamma_1,\dots, \gamma_h,\omega\}$. 
Changing the basis of $H_w$ by adding to $\omega$ a linear combination of $\gamma_1,\dots, \gamma_h$ does not change the  determinant of 
$\langle\,.\,,.\, \rangle$, so $\psi(H_w, \underline \ovar)$ only depends on $w$ and the basis $\{\gamma_1, \ldots, \gamma_h\}$ of $H$. 
\end{proof} 
 
\begin{defn}\label{second_symanzik} The second Symanzik polynomial associated to $H$, $w\in \ws$, and the variables $\underline \ovar=\{\ovar_e\}_{e\in E}$, is the polynomial 
\begin{equation*}
\phi(H,w,\underline \ovar)=\psi(H_w, \underline \ovar).
\end{equation*}
\end{defn}

\begin{prop} The ratio $\phi(H,w, \underline \ovar) \slash \psi(H,
  \underline \ovar)$ between the first and the second Symanzik
  polynomials does not depend on the choice of a basis of $H$. 
\end{prop}

\begin{proof} Let $\langle\,.\,,.\, \rangle$ still denote the bilinear form on $H_w$ associated to the quadratic form $\sum \ovar_e e^{\vee,2}$. Then $\psi(H, \underline \ovar)$ is the determinant, in the given basis, of the restriction of $\langle\,.\,,.\, \rangle$ to $H$. 
Changing the basis multiplies both $\psi(H, \underline \ovar)$ and $\phi(H,w, \underline \ovar)=\psi(H_w, \underline \ovar)$ by the same factor in $\ok^{\times, 2}$, from which the claim follows. 
\end{proof}

To give another formula for the second Symanzik polinomial which will be used later we
introduce the following notation.

\begin{defn}\label{def:1} Let $w \in \ws \setminus \{0\}$ and
  choose $\omega \in \K^{E}$ in the preimage of $w$. We denote by
  $W_{e}(\omega )$ the column vector with components $\langle \gamma
  _{i},\omega \rangle_{e}$. If $w'$ and $\omega '$ is another choice of
  such vectors we write
  \begin{displaymath}
    Q_{e}(\omega ,\omega ')=\langle \omega ,\omega '\rangle_{e} \quad \text{and} \quad Q_{e}(\omega)=Q_{e}(\omega ,\omega).
  \end{displaymath}
\end{defn}

\begin{prop}\
  \begin{enumerate}
  \item The first Symanzik polynomial $\psi(H,\underline \ovar)$ is
    homogeneous of degree $h=\dim H$ in the variables $\ovar_e$. 
  \item The second Symanzik is given by
    \begin{equation}
      \label{eq:4}
      \phi(H,w,\underline \ovar)=\det\Big(\sum
      \ovar_e\begin{pmatrix}M_e & W_e(\omega ) \\ 
\prescript{\mathrm t}{}{W_e(\omega )} & Q_e(\omega )\end{pmatrix}\Big).
    \end{equation}
    Moreover, it is homogeneous of degree $h+1$ in the variables
    $\ovar_e$ and is quadratic in $w \in \ws \setminus \{0\}$.  
  \end{enumerate}
\end{prop}

\begin{proof} The first statement is clear from
  $\psi(H,\underline \ovar)=\det(\sum Y_{e}M_{e})$. Equation \eqref{eq:4} is
  just a reformulation of the definition of the second Symanzik
  polynomial, from which the last statement follows immediately. 
  \end{proof}

 One can slightly generalize the definition of the second Symanzik polynomial. Let $\vs$ be a vector space over $\ok$ equipped with a quadratic form $q$ associated to a symmetric bilinear form $\langle\cdot\,,\cdot\rangle_q$. 
 Using it, one can make sense of the determinant on the right hand side of \eqref{eq:4} and define $\phi(H,w,\underline \ovar)$ for any nonzero element $w\in \ws\otimes_\ok \vs$.  
 Typically, for physics applications, $\ok=\R$, $\vs=\R^D$ is space-time, and $q$ is the Minkowski metric. This works as follows. 
 
 \smallskip
 
 One naturally extends
 $\langle\,.\,,.\,\rangle_e$ to a 
 bilinear pairing 
 \begin{equation*}
 \langle\,.\,,.\,\rangle_e: H \times ( \ok^E \otimes_\ok \vs) \rightarrow \vs,
 \end{equation*} and $e^{\vee,2}$ to the quadratic form $Q_e$ on $\ok^E\otimes_{\ok} \vs$ given by 
 \begin{equation*}
 Q_e(\alpha \otimes \beta) = e^{\vee,2}(\alpha) q(\beta).
 \end{equation*}
 
 \smallskip
 
 Let $w\in \ws\otimes_\ok \vs$ and consider $\omega \in \ok^E\otimes \vs$ in the preimage of $w\in\ws\otimes_\ok \vs$. Applying the bilinear pairing to $(\gamma_i, \omega)$ leads to the column vector 
 $W_e=(v_{e,1},\dotsc,v_{e,h})$ with entries $v_{e,i}\in \vs$.
 
 \smallskip
 
 Using the above extension, one can now make sense of the determinant in \eqref{eq:4} and define $\phi(H,w,\underline \ovar)$ for $w\in \ws\otimes \vs$. To explain this, 
 suppose we have a symmetric $(h+1)\times (h+1)$ matrix of the form 
\begin{equation*}
T= \begin{pmatrix} M & W\\
\prescript{\mathrm t}{}{W} &  S
\end{pmatrix}
\end{equation*}
where $M$ is an invertible $h\times h$ matrix, $W$ is a (column) vector of dimension $h$, and $S$ is a scalar. The formula $(\det M)M^{-1} = \text{adj}(M)$, where $\text{adj}(M)$ is the matrix of minors, gives 
\begin{equation*}
\frac{\det T}{\det M}= -\prescript{\mathrm t}{}{W}M^{-1}W+S.
\end{equation*}

\smallskip

Taking $W = \sum_{e\in E} \ovar_e W_e$, where $W_e$ has now entries in $\vs$, the determinant in \eqref{eq:4} can be written as
\begin{equation}\label{eq:21}
\frac{\phi(H,w,\underline \ovar)}{\psi(H, \underline \ovar )} = -\prescript{\mathrm t}{}{W}M^{-1}W+Q(\omega),
\end{equation}
where $Q(\omega) = \sum_{e\in E}\ovar_e Q_e(\omega)$, and the product $\prescript{\mathrm t}{}{W}M^{-1}W$ is interpreted via the bilinear form 
(in the sense that, developing the product as the sum of the form $\sum v_{e,i} m_{i,j} v_{e,j}$, with $v_{e,i}, v_{e,j} \in \vs$, $m_{i,j}\in \ok[\underline\ovar]$, becomes $\sum m_{i,j} \langle v_{e,i},v_{e,j}\rangle_q$).

\smallskip

The expression \eqref{eq:21} will be
  later used to relate the second Symanzik to the archimedean height.

\subsection{Graphs}
In what follows, $G$ is a connected graph with edge set $E=E(G)$ and vertex set $V=V(G)$. We will fix an orientation on the edges so we have a boundary map $\partial: \Z^E \to \Z^V,\ e \mapsto \partial^+(e)-\partial^-(e)$, 
where $\partial^+$ and $\partial^-$ denote the head and the tail of
$e$, respectively. The homology of $G$ is defined via the exact
sequence
\begin{equation}\label{eq:6}
0 \rightarrow H_1(G,\Z) \rightarrow \Z^E \xrightarrow{\partial} \Z^V
\rightarrow \Z \rightarrow 0. 
\end{equation}
Homology with coefficients in any abelian group is defined similarly. 

\smallskip

In order to apply the constructions of the previous section, we write $H=H_{1}(G,\K)$. The exact sequence \eqref{eq:6} yields an isomorphism
\begin{equation}
  \label{eq:20}
\ws=\K^{E}/H \simeq \K^{V,0}, 
\end{equation}
where $\K^{V, 0}$ consists of those $x \in \K^V$ whose coordinate sum to zero. We will use \eqref{eq:20} to identify both spaces.

\begin{defn} The first Symanzik polynomial of a connected graph $G$ is the first Symanzkik polynomial, as in Definition \ref{def:1st}, associated to the configuration $H=H_1(G, \ok )\subset \ok^E$. We will denote it by 
$$
\psi_G(\underline \ovar)=\psi(H, \underline \ovar). 
$$
\end{defn}

Since $H=H_1(G, \Z) \otimes_{\Z} \ok$, Remark~\ref{rem:ind} guarantees that $\psi_G$ is independent of the choice of an integral basis. The link between the above definition and the expression for $\psi_G$ given in the introduction is the content of Kirchhoff's matrix-tree theorem~\cite{Kir}. Recall that a subgraph $T$ of $G$ is called a spanning tree if it is connected and simply connected, and satisfies $V(T) = V(G)$. %and $H_1(T,\Z)=0$. Note that this definition makes sense even when the graph $G$ is not connected.

\begin{prop} The first Symanzik polynomial $\psi_G$ is equal to
\begin{equation*}
\psi_G (\underline \ovar) = \sum_{T\subset G}\prod_{e\not\in T}\ovar_e,
\end{equation*}
 where $T$ runs through all spanning trees of $G$. 
\end{prop}

The second Symanzik polynomial can be described explicitly via the \emph{external momenta}, as we explain now. 
Note that, in the situations coming from physics, it also depends on the masses associated to the edges. However, in this paper we only consider the massless case. 

\smallskip

Let $\vs$ be a vector space over a field $\ok$ with a symmetric bilinear form $\langle\cdot\,,\cdot\rangle_q$, and consider $\omega \in \ok^E \otimes \vs$ 
which reduces to the vector $w \in \ws \otimes \vs$, where $H = H_1(G, \ok)$ and $\ws = \ok^E /H$. Using the isomorphism \eqref{eq:20}, we have $\ws\otimes \vs  \simeq
\ok^{V,0}\otimes \vs \simeq \vs^{V,0}$. In other words, the choice of
$w \in \ws \otimes \vs$ is equivalent to the choice of  external
momenta $\ps_v\in \vs$ satisfying
the conservation law $\sum_{v \in V} \p_v=0.$

\begin{defn} Let $G$ be a connected graph and $\underline \p = \{\ps_v\} \in \vs^{V,0}$ be external momenta. The second Symanzik polynomial of $(G, \underline \p)$ is 
\begin{equation*}
\phi_G(\underline \p, \underline \ovar)=\phi(H, \omega, \underline \ovar)
\end{equation*} for the element $\omega \in \ok^E \otimes \vs$ with $\p=\partial(\omega)$.
\end{defn}

To give an explicit description of the polynomial $\phi_G(\underline \p, \underline \ovar)$, we need to introduce some extra notation. Let $G$ be a connected graph. 
A spanning $2$-forest $F\subset G$ is a subgraph of $G$, with two connected components $F_1$ and $F_2$, satisfying $V(F) = V(G)$ 
and $H_1(F,\Z)=0$ (so each $F_i$ is a subtree of $G$). Given a collection of external momenta $\underline \p = (\ps_v)\in \vs^{V,0}$ and a spanning 2-forest, we define $\ps(F_i)= \sum_{v\in V(F_i)} \ps_v$, the total momentum entering $F_i$, and $q(F)=- \langle \ps(F_1), \ps(F_2)\rangle_q = q(\ps(F_1))$, where the last equality follows from the conservation law. Then we have the following proposition, for which we refer the reader e.g. to~\cite{Chaiken}:

\begin{prop} \label{prop:second}
If $G$ is a connected graph, then 
\[\phi_G(\underline \p, \underline \ovar) = \sum_{F \subset G} q(F)\prod_{e\notin E(F)} \ovar_{e}\,,\]
where the sum runs over all spanning $2$-forests $F$ of $G$. 
\end{prop}

\section{Degeneration of curves}\label{riemannsurf}

The aim of this section is to interpret the rank one symmetric matrices $M_e$ introduced in \eqref{eq:1} in terms of the monodromy of a degenerating family of curves~\cite{CCK, C}. For this, we fix a complex stable curve $C_0$ of arithmetic genus $g$ and dual graph $G=(V, E)$. Throughout $h$ denotes the first Betti number of $G$.  

\smallskip

Concretely, $C_0$ is a projective connected nodal curve with smooth irreducible components $X_v$ indexed by the vertices of $G$. It is obtained as a quotient of $\coprod_{v \in V} X_v$ by identifying a chosen point of $X_v$ with a chosen point of $X_w$ whenever there exists an edge connecting $v$ and $w$. Stability means that the automorphism group of $C_0$ is finite: letting $g(X_v)$ denote the geometric genus of $X_v$ and $\mathrm{val}(v)$ the valency of a vertex, this is equivalent to $2g(X_v)-2+\mathrm{val}(v)>0$ for every $v \in V$. 

\begin{prop}\label{genus} The identification map $p: \coprod_{v \in V} X_v \to C_0$ induces a canonical isomorphism
$$
H^1(G,\C) \simeq \ker\Bigl(\, H^1(C_0, \mathcal O_{C_0}) \stackrel{p^\ast}{\rightarrow} \bigoplus_{v\in V} H^1(X_v, \mathcal O_{X_v})\,\Bigr),
$$
and the arithmetic genus of $C_0$ is equal to $h+\sum_{v\in V}g(X_v)$, where $h$ is the first Betti number of $G$. Moreover, 
$$
H^1(G,\Z)\simeq \ker\Bigl(H^1(C_0,\Z) \stackrel{p^\ast}{\rightarrow} \bigoplus_{v \in V} H^1(X_v, \Z)\Bigr).$$ 
\end{prop}

\begin{proof} Let us choose an orientation of the edges of $G$. Then we have an exact sequence of sheaves
\begin{equation}\label{eq:34}
0 \to \sO_{C_0} \to p_\ast\sO_{\coprod X_v} \stackrel{\varphi}{\longrightarrow} \sS \to 0,
\end{equation}
where $\mathcal S$ is a skyscraper sheaf with stalk $\C$ over each singular point of $C_0$ and the map $\varphi=(\varphi_e)_{e \in E}$ is defined as follows: if $f$ is a local section of $p_\ast\sO_{\coprod X_v}$ near the singular point corresponding to $e$, then $\varphi_e(f)=f(P_v)-f(P_w)$ where $v$ and $w$ denote the head and the tail of $e$ respectively, and $P_v \in X_v$ and $P_w \in X_w$ are the points identified to get $C_0$. Observe that, since $p$ is finite, taking cohomology commutes with $p_\ast$, so we get the exact sequence
\begin{equation*}
0 \rightarrow \C \rightarrow \C^{V} \xrightarrow{\delta} \C^{E} \rightarrow H^1(C_0, \sO_{C_0}) \stackrel{p^\ast}{\longrightarrow} \bigoplus_{v \in V} H^1(X_v, \sO_{X_v})\to 0, 
\end{equation*}
where $\delta$ is dual to the boundary map (with respect to the same orientation of the edges) in the definition \eqref{eq:6} of the graph homology. Thus, $H^1(G, \C) \simeq \mathrm{coker}(\delta)$ and the first isomorphism, as well as the expression for the arithmetic genus of $C_0$, follows.

\smallskip

The proof of the second assertion goes in the same way, up to replacing the exact sequence \eqref{eq:34} by the analogous sequence of constructible sheaves calculating Betti cohomology. 
\end{proof}

\subsection{Deformations}\label{sec-def}

We recall some basic facts about deformation theory of stable curves, for which we refer the reader to \cite{DM69, HM, Hart, Sch}. 

\smallskip

Let $C_0$ be, as before, a complex stable curve of arithmetic genus $g$. Standard results in deformation theory provide a smooth formal scheme $\widehat{S} = \text{Spf }\C[[t_1,\dotsc,t_N]]$ and a versal formal family of curves $\pi \colon \widehat{\sC} \to \widehat{S}$ with a fibre $\sC_0$ over $0\in \widehat{S}$ isomorphic to $C_0$. 
In particular, the total space $\widehat{\sC}$ is formally smooth over $\C$, and we get an identification of the tangent space $T$ to $\widehat{S}$ at $0$ with the Ext group $\Ext^1(\Omega^1_{C_0}, \sO_{C_0})$.  
Locally (for the \'etale topology) at the singular points, we have
\begin{equation}\label{loc-pre}
C_0 \simeq \Spec R, \quad R=\C[x,y]/(xy),
\end{equation}
so $\Omega^1_{C_0} \simeq Rdx\oplus Rdy/(xdy+ydx)$. In particular, since $xdy \in \Omega^1_{C_0}$ is killed by both $x$ and $y$, it follows that $\Omega^1$ has a non-trivial torsion subsheaf supported at the singular points. 
By the vanishing of the higher degree terms~\cite{DM69}, we get the following short exact sequence from 
the five term exact sequence of low degree terms in the local to global Ext-spectral sequence 
\begin{equation}\label{eq:5term}
0 \to H^1(C_0, \usHom(\Omega^1_{C_0},\sO_{C_0})) \to T \to \Gamma(C_0, \usExt^1(\Omega^1_{C_0},\sO_{C_0})) \to 0.
\end{equation}

The local $\sExt$ sheaf on the right can be easily calculated using
the local presentation \eqref{loc-pre} at the singular points: 
\begin{align*}
0 \to R &\longrightarrow Rdx\oplus Rdy \to \Omega^1_R \to 0 \\
1 &\longmapsto xdy+ydx 
\end{align*} One identifies in this way $\usExt^1(\Omega^1_{C_0},\sO_{C_0})$ with
the skyscraper sheaf consisting of one copy of $\C$ supported at each
singular point, hence
\begin{equation}\label{eq:sky}
\Gamma\bigl(C_0, \usExt^1(\Omega^1_{C_0},\sO_{C_0})\bigr) \simeq \C^E.
\end{equation}

\smallskip

In addition, $\Gamma\bigl(C_0,
\usExt^1(\Omega^1_{C_0},\sO_{C_0})\bigr)$ corresponds to smoothings of
the double points~\cite{DM69}, so the   
subspace $H^1(C_0,  \usHom(\Omega^1_{C_0},\sO_{C_0}))\subset T$
corresponds to deformations such that
only $X_v$ with the points corresponding to incident edges to
$v$ in $G$ move.  
Since $r$ points on $X_v$ of genus $g(X_v)$ have $3g(X_v)-3 + r$
moduli, we get the dimensions 
\begin{equation*}
\dim H^1(C_0,  \usHom(\Omega^1_{C_0},\sO_{C_0})) = \sum_{v\in
  V} (3g(X_v) -3 + \mathrm{val}(v)), 
\end{equation*}
\begin{align*}
\dim T &= \sum_{v\in V} (3(g(X_v) + \mathrm{val}(v)-3) +|E| \\
&=-3|V| +3|E|+ \sum_{v \in V} 3 g(X_v) \\
&= 3g-3, 
\end{align*} where the last equality follows from Proposition \ref{genus}. Note that the arithmetic genus of $C_0$ coincides with the usual genus of a smooth deformation of $C_0$, and that $3g-3$ is also the dimension of the moduli space of curves of genus $g$. 

\smallskip

For each edge $e\in E$, let  $p_e \in C_0$ be the corresponding singular point of $C_0$. Those deformations of $C_0$ which preserve the singularity at $p_e$ are given by a divisor $\widehat{D}_e \subset \widehat{S}$. Suppose that $f_e \in \sO_{\widehat{S}}$ defines $\widehat{D}_e$. Then the functional $T \to \C^E \xrightarrow{pr_e} \C$ is defined by $df_e$. Taking into account the identification \eqref{eq:sky}, this yields the surjective map $T \to \C^E$ in \eqref{eq:5term}. In the geometric picture, we have a collection of principal divisors $\widehat{D}_e\subset \widehat{S}$ indexed by the edges of $G$ which meet transversally. The subvariety cut out by these divisors is precisely the locus of equisingular deformations of $C_0$ 
which are given by moving the singular points. 

\smallskip

Similarly, if $(C_0, q_1, \ldots, q_n)$ is a complex stable curve of arithmetic genus $g$ with $n$ marked points, there exists a formal disc $\widehat{S}'$ of dimension $3g-3+n$ and a versal formal deformation $\pi \colon \widehat{\mathcal{C}}' \to \widehat{S}'$ such that the tangent space to $0 \in \widehat{S}'$ is identified with $\Ext^1(\Omega^1_{C_0}, \sO_{C_0}(-\sum_{i=1}^n q_i))$. The fibre at $0$ is isomorphic to $C_0$, and the family comes together with sections $\sigma_i \colon \widetilde{S}' \to \widetilde{\mathcal{C}}'$ such that $\sigma_i(0)=q_i$.

\subsection{Monodromy} The formal schemes given by the deformation theory can be spread out to yield an analytic deformation $\sC \to S$, where $S$ is a polydisc of dimension $3g-3$. 
In this way, the divisors lift to analytic divisors $D_e \subset S$ which are defined by the equation $\{f_e=0\}$. 

\smallskip

We fix a basepoint $s_0 \in S \setminus \bigcup_{e \in E} D_e$. 
The goal is to study the monodromy action on $H_1(C_{s_0},\Z)$. For this, we choose, for each $e\in E$, a simple loop $\ell_e\subset S \setminus \bigcup_{e \in E} D_e$ based at $s_0$ which loops around the divisor $D_e$. 
We assume that $\ell_e$ is contractible in the space $S \setminus \bigcup_{\ve\neq e}D_\ve$.  

\smallskip

The monodromy for the action of $\ell_e$ on $H_1(C_{s_0},\Z)$ is given by the {\it Picard-Lefschetz formula}: 
\begin{equation}\label{eq:42}
\beta \mapsto \beta-\langle \beta,a_e\rangle a_e,
\end{equation}
where $a_e\in H_1(C_{s_0},\Z)$ denotes the vanishing cycle associated to the double point which remains singular as one deforms along $D_e$. 
\smallskip

By a basic result in differential topology, after possibly shrinking the polydisc $S$, the inclusion $C_0 \hookrightarrow \sC$ admits a retraction $\sC \rightarrow C_0$ 
in such a way that the composition $\sC \to C_0 \to \sC$ becomes homotopic to the identity. The inclusion $C_0 \inj \sC$ is thus a homotopy equivalence, and from this, one gets the specialization map
\begin{equation*}
\mathrm{sp} \colon H_1(C_{s_0},\Z) \rightarrow H_1(\sC,\Z) \simeq H_1(C_0,\Z).
\end{equation*}

\begin{lem}The specialization map $\mathrm{sp}$ above is surjective. 
\end{lem}

One can give a formal proof based on the Clemens-Schmid exact
sequence, see e.g. \cite{Mor}. Intuitively, a loop in $H_1(C_0,\Z)$
can be broken up into segments which connect double points of the
curve. Since these double points arise by shrinking (vanishing) cycles
on $C_{s_0}$, we can model the segments by segments in $C_{s_0}$ which
connect the vanishing cycles.  Connecting all these segments together
yield a loop in $C_{s_0}$ which specializes to the given loop in
$C_0$.

\smallskip

Let $A\subset H_1(C_{s_0},\Z)$ denote the subspace spanned by the
vanishing cycles $a_e$. Observe that we have an exact sequence
\begin{equation*}
0\to A \to H_1(C_{s_0}, \Z) \xrightarrow{\mathrm{sp}} H_1(C_0, \Z) \to 0.
\end{equation*}

Define $A'=A + \mathrm{sp}^{-1}(\bigoplus_{v \in V} H_1(X_v, \Z))\subseteq H_1(C_{s_0}, \Z)$. Using Proposition \ref{genus}, we have
\begin{equation}
 \label{eq:aux}
H_1(C_{s_0}, \Z)/A' \simeq H_1(C_0, \Z) / \bigoplus_{v \in V} H_1(X_v, \Z)\simeq H_1(G, \Z).
\end{equation}

\begin{lem}\label{lem_isotropic} The subspace $A\subset
  H_1(C_{s_0},\Z)$ defined by the vanishing cycles is isotropic and
  has rank $h$. In particular, it is maximal isotropic if the genera
  of all the components $X_v$ are zero.
\end{lem}
\begin{proof} For $s_0$ very close to $0\in S$, the vanishing cycles
  $a_e$ become disjoint since they approach different singular points
  $p_e\in C_0$.  Thus, $\langle a_e, a_{e'}\rangle = 0$. The pairing
  on $H_1$ being symplectic, we automatically have
  $\langle a_e,a_e\rangle=0$, which shows that the subspace $A$ is
  isotropic.

  To prove the claim about the dimension, note that \eqref{eq:aux}
  implies
  \begin{align*}
    h&=\rk H_1(C_0, \Z) - 2\sum_{v\in V}g(X_v)=\rk H_1(C_{s_0}, \Z) - \rk A- 2\sum_{v\in V}g(X_v)\\ &=2g(C_{s_0})-\rk A- 2\sum_{v\in V}g(X_v)= 2h -\rk A. 
  \end{align*}
  It follows that $\rk A = h$.
\end{proof}

The same reasoning as above implies that $\langle A,A'\rangle =0$,
so the symplectic pairing reduces to a pairing
$A \times \Bigl(H_1(C_{s_0}, \Z)/A'\Bigr) \rightarrow \Z$.

\smallskip

Write $N_e = \ell_e-\id$. Note that by \eqref{eq:42}, we have
$N_e(\beta)=\langle \beta,a_e\rangle a_e$ for any
$\beta \in H_1(C_{s_0}, \Z)$. Thus the image of $N_{e}$ is contained
in $A$. By Lemma~\ref{lem_isotropic}, we get
that $N_{e}$ vanishes in $A$. Thus 
$N_e^2=0$ which shows that $N_e=\log(\ell_e)$. Consider now the
composition
\begin{equation*}
H_1(C_0, \Z) \simeq H_1(C_{s_0}, \Z)/A \xrightarrow{N_e} A \simeq
(H_1(C_{s_0}, \Z)/A')^\vee \simeq H_1(G, \Z)^\vee.
\end{equation*}
Note that, by Picard-Lefschetz, all the elements in
$$ 
\mathrm{sp}^{-1}(\bigoplus_v H_1(X_v, \Z))\subset H_1(C_{s_0}, \Z)
$$ 
are in the kernel of $N_e$, so in fact the above map passes to the
quotient to give a map
\begin{equation}\label{45bis}
  H_1(G, \Z) \simeq H_1(C_{s_0}, \Z)/A' \xrightarrow{N_e} H_1(G, \Z)^\vee.
\end{equation}

\smallskip 
The following proposition provides the relation between the
monodromy and the combinatorics of the graph polynomials.

\begin{prop}The bilinear form on $H_1(G, \Z)$ given by \eqref{45bis}
  coincides with the bilinear form $\langle\cdot\,,\cdot\rangle_e$.
\end{prop}
\begin{proof} For any $b\in H_1(C_{s_0}, \Z)$, the image of
  $\mathrm{sp}(b)$ in the quotient
  $H_1(C_{s_0}, \Z)/A' \simeq H_1(G, \Z)$ can be identified with a
  loop $\gamma=\sum_e n_ee$, where $n_e= \langle b,a_e\rangle$ denotes
  the multiplicity of intersection of $b$ with the vanishing cycle
  $a_e$. The quadratic form on $H_1(C_{s_0}, \Z)$ associated to $N_e$
  sends $b$ to $\langle b, \langle b,a_e\rangle a_e\rangle =
  n_e^2$. On the other hand, the bilinear form
  $\langle\cdot\,,\cdot\rangle_e$ on $H_1(G, \Z)$ corresponding to the
  edge $e$ sends the loop $\gamma$ to $n_e^2$, from which the
  proposition follows.
\end{proof}

 Fix a 
 symplectic basis $a_1,\dotsc,a_{g},b_1,\dotsc,b_{g}$ for
 $H_1(C_{s_0}, \Z)$ such that $a_1, \dots, a_h$ form a basis of
 $A$. For any $e\in E$, we can write  
\begin{equation}\label{45a}
a_e = \sum_{i=1}^h c_{e,i}a_i.
\end{equation}
 
Let $B \subset H^1(C_{s_0}, \Z)$ be the subspace generated by $b_1,
\dots, b_h$. Note that we have isomorphisms $B \simeq A^{\vee} \simeq
H_1(G, \Z)$. Thus, we can see the monodromy operators as maps
\begin{displaymath}
  N_{e}\colon A \longrightarrow B.
\end{displaymath}
The following proposition is straightforward. 
%In terms of the basis $b_1,\dots, b_g$ for $B\simeq H_1(G)$ 

\begin{prop}\label{rmk25} In terms of the basis $b_1,\dots, b_h$ for $B\simeq H_1(G, \Z)$, we can write $M_e=\Bigl(c_{e,i}c_{e,j}\Bigr)_{1\leq i,j\leq h}$.
%using the notation of \eqref{45a}.
\end{prop}

We will denote by $A_{0}$ (respectively $B_{0}$) the subspace of
$H^1(C_{s_0}, \Z)$ generated by $a_{1},\dots ,a_{g}$ (respectively
$b_{1},\dots,b_{g}$). The $A_{0}$ is a maximal isotropic subspace
with $A\subset A_{0}\subset A'$.

\section{Archimedean heights and the Poincar\'e bundle}
\label{sec:arch-heights-poinc}

In this section, we recall the definition of the archimedean height pairing between degree zero divisors with disjoint support on a smooth projective curve, as well as its interpretation in terms of biextensions and the Poincar\'e bundle. Throughout, given a set $\Sigma$ and a ring $R$, we denote by $(\bigoplus_{\Sigma} R )^0$ the set of elements $(r_s) \in \bigoplus_{\Sigma} R$ with $\sum_s r_s =0$. 

\subsection{Archimedean heights}
\label{sec:archimedean-heights}
Let $C$ be a smooth projective curve over the field of complex numbers and $\Sigma \subset C$ a finite set of points in $C$, which we also think of as a reduced effective divisor. The inclusion $j \colon C \setminus \Sigma \hookrightarrow C$ yields an exact sequence of mixed Hodge structures:
\begin{equation}\label{r-mix}
0 \to H^1(C, \Z(1)) \stackrel{j^\ast}{\longrightarrow} H^1(C \setminus \Sigma, \Z(1)) \longrightarrow (\bigoplus_{\Sigma} \Z )^0 \to 0.
\end{equation}

\begin{lem}\label{lemm:1} The exact sequence of real mixed Hodge
  structures obtained from \eqref{r-mix} by tensoring with $\R$ is
  canonically split.  
\end{lem}

\begin{proof} It suffices to show that an extension of real mixed
  Hodge structures of the form $0 \to H \to E \xrightarrow{\alpha}
  \R(0) \to 0$, where $H$ is pure of weight $-1$, is canonically
  split. For this, we consider the subspace $M=F^0 E_{\C} \cap
  \overline{F^0 E_{\C}}$ of $E$. The map $\alpha$ induces a surjection
  $M \to \R$ with kernel $M \cap H$. Since $H$ has weight $-1$, this
  intersection is empty. We thus get an isomorphism whose inverse
  map provides the splitting. 
\end{proof}

Recall that the Hodge filtration on $H^1(C \setminus \Sigma, \C)$
comes from the exact sequence of sheaves  
\begin{equation}\label{poin-res}
0 \to \Omega^1_{C} \xrightarrow{}  \Omega^1_{C}(\log \Sigma)
\xrightarrow{\Res_{\Sigma }} \bigoplus_{\Sigma} \underline{\C} \to 0,  
\end{equation} 
were $\Res_{\Sigma }=\sum_{p\in \Sigma }\Res_{p}$, and $\Res_{\Sigma}$ is defined, for a local section $\omega$, by
\begin{displaymath}
  \Res_{p}(\omega )=\frac{1}{2\pi i}\int_{\gamma _{p,\varepsilon }}\omega ,
\end{displaymath}
for a small negatively oriented circle $\gamma _{p,\varepsilon }$
around the point $p$. We take $\gamma _{p,\varepsilon }$ negatively
oriented because we want it to be part of the bounday of the
complement of a small disk around $p$ and not of the disc itself. 

\smallskip

From the exact sequence \eqref{poin-res} we deduce that $F^0 H^1(C
\setminus \Sigma, 
\C(1))=H^0(C, \Omega^1_C(\log \Sigma))$. Combining this information
with Lemma~\ref{lemm:1}, we get a canonical map
\begin{align*}
(\bigoplus_{\Sigma} \R )^0 &\longrightarrow H^0(C, \Omega^1_{C}(\log
\Sigma))\cap H^1(C \setminus \Sigma,
\R(1)) \\
\oD &\longmapsto \omega_{\oD,\R}.
\end{align*}

\begin{remark}\label{omega-uniq} Concretely, $\omega_{\oD,\R}$ can be
  understood as follows. By \eqref{poin-res}, the condition
  $\textrm{Res}_{\Sigma} \omega_{\oD}=\oD$ determines a logarithmic
  differential $\omega_{\oD}$
  only up to addition of elements in $H^0(C, \Omega^1_C)$. To fix it
  uniquely, we require that $\int_\gamma \omega_{\oD,\R} \in \R(1)$ for
  every real valued cycle $\gamma$ in $C \setminus \Sigma$.  Note that $\omega
  _{\oD,\R}$ is an ``admissible integral'' in the sense of
  \cite[Def. 3.3.5]{Hain}.
\end{remark}

Let $\mathfrak A$ be a degree zero $\R$-divisor on $C$ with support
$\Sigma$ and let $\omega _{\oA,\R}$ the form just defined. Given
another degree zero $\R$-divisor $\mathfrak B$ with 
disjoint support, we can find a real-valued $1$-chain
$\gamma_{\mathfrak{B}}$ on $C \setminus \Sigma$ such that $\mathfrak B
= \partial \gamma_{\mathfrak{B}}$.

\begin{defn} The archimedean height pairing between $\oA$ and
  $\oB$ is the real number 
\begin{equation}\label{height-def}
\langle \oA, \oB
  \rangle=\mathrm{Re}\left(\int_{\gamma_{\oB}} \omega_{\oA,\R}\right).
\end{equation}
\end{defn}

Note that, since $\omega_{\oA, \R}$ is an $\R(1)$-class, modifying
$\gamma_{\mathfrak B}$ by an element of $H_1(C \setminus \Sigma,\R)$
does not change the real part of the integral. Therefore the above
definition is independent of the choice of $\gamma_{\mathfrak
  B}$. Though not apparent from \eqref{height-def}, the archimedean
height pairing is symmetric.

\begin{ex}\label{ex:heigh-div} When the divisor $\mathfrak{A}$ is of
  the form $\mathrm{div}(f)$ for a rational function $f$ on $C$, the
  differential $\omega_{\mathfrak{A}, \R}$ is nothing else than
  $-\frac{df}{f}$, hence 
\begin{equation}
\langle \mathfrak{A}, \mathrm{div}(f) \rangle=\mathrm{Re} \left(\int_{\gamma_{\oB}} -d\log |f| \right)
=-\log |f(\mathfrak{B})|. 
\end{equation} 
\end{ex}

\smallskip

Finally, consider the case of divisors with values in space-time
$\R^D$ with a given Minkowski metric. Tensoring with $\R^D$ and using
the Minkowski metric, the archimedean height pairing extends to a
pairing between 
degree zero $\R^D$-valued divisors with disjoint support.

\subsection{Biextensions}\label{sec:biext} 

Let $\mathfrak A$ and $\mathfrak B$ be integer-valued degree zero divisors with disjoint supports $|\mathfrak A|$ and $|\mathfrak B|$ on $C$. In this paragraph, we recall how to attach to $\mathfrak A$ and $\mathfrak B$ a mixed Hodge structure $M$ with weights $-2, -1, 0$ and graded pieces
$$
\gr_{-2}^W M=\Z(1), \quad \gr^W_{-1}M=H^1(C, \Z(1)), \quad \gr_0^WM=\Z(0).
$$ Such mixed Hodge structures are called \textit{biextensions}. The standard reference is Section 3 of Hain's paper \cite{Hain}. 

\smallskip

We first observe that pulling back the exact sequence \eqref{r-mix} by the map $\Z \rightarrow \left(\bigoplus_{|\mathfrak A|} \Z\right)^0$ which sends $1$ to the divisor $\mathfrak{A}$, we get an extension 
$$
H_{\mathfrak{A}} \in \Ext^1_{\mathrm{MHS}}
\Bigl(\,\Z(0),H^1(C,\Z(1))\,\Bigr)
$$ which fits into a diagram 
\begin{equation}\label{ext-diag}
\xymatrix{ 0 \ar[r] & H^1(C, \Z(1))\ar[r] \ar@{=}[d] & H^1(C\setminus |\mathfrak A|, \Z(1)) \ar[r] & \left(\bigoplus_{|\mathfrak{A}|} \Z\right)^0 \ar[r] & 0\\ 0 \ar[r] & H^1(C, \Z(1))\ar[r] & H_{\mathfrak{A}} \ar[r] \ar[u]^{\gamma} & \Z(0) \ar[u] \ar[r] & 0.
}
\end{equation}
Abusing notation, $H_{\mathfrak{A}}$ is also denoted by $H^1(C\setminus \mathfrak A, \Z(1))$.

\smallskip

Similarly, from the cohomology of $C$ relative to $|\mathfrak B|$ we obtain an exact sequence of mixed Hodge structures 
\begin{equation}\label{relative}
0\rightarrow \mathrm{coker}\Bigl( \Z \to \bigoplus_{|\mathfrak B|} \Z\Bigr)\rightarrow H^1(C, |\mathfrak B|;\Z) \rightarrow H^1(C, \Z) \to 0.
\end{equation}
Pushing forward \eqref{relative} by the map $\mathrm{coker}\Bigl(\,\Z \to \bigoplus_{|\mathfrak B|}\Z\,\Bigr) \to \Z$ given by the coefficients of $\mathfrak B$ and tensoring by $\Z(1)$, we get an element 
$$
H^1(C, \mathfrak B;\Z(1)) \in \Ext^1_{MHS}\Bigl(H^1(C, \Z(1))\,,\,\Z(1) \Bigr).
$$

\begin{remark} \rm Applying $\Hom_{MHS}(-, \Z(1))$ to this extension and using Poincar\'e duality we get $H_{\mathfrak{B}}$. 

\end{remark}

\smallskip

Since $\mathfrak{A}$ and $\mathfrak{B}$ have disjoint support, replacing $C$ by $C \setminus | \mathfrak{A} |$ in \eqref{relative} and proceeding as before yields another extension 
\begin{equation}\label{ext3}
0 \to \Z(1) \to E \to H^1(C \setminus |\mathfrak{A} |, \Z(1)) \to 0. 
\end{equation}

\begin{defn} The biextension mixed Hodge structure associated to $\mathfrak{A}$ and $\mathfrak{B}$ is the pullback of the extension \eqref{ext3} by the map $\gamma$ in \eqref{ext-diag}. It will be denoted either by $H_{\mathfrak{B}, \mathfrak{A}}$ or by $H^1(C \setminus \mathfrak A, \mathfrak B; \Z(1))$. 
\end{defn}

By construction, $H_{\mathfrak{B}, \mathfrak{A}}$ fits into the diagram  
\begin{equation}\label{eq:70}
\xymatrix{ & & 0 \ar[d] & 0 \ar[d] & \\ 0 \ar[r] & \Z(1) \ar@{=}[d] \ar[r] & H^1(C,\mathfrak{B};\Z(1)) \ar[d] \ar[r] & H^1(C, \Z(1)) \ar[d] \ar[r]& 0\\
0 \ar[r] & \Z(1) \ar[r] & H_{\mathfrak{B}, \mathfrak{A}} \ar[d] \ar[r] & H_{\mathfrak{A}}  \ar[d] \ar[r]& 0 \\
& & \Z(0) \ar@{=}[r]\ar[d] & \Z(0) \ar[d] & \\
& & 0 & 0 &
}
\end{equation}

In particular, the weight filtration is given by 
\begin{equation*}
0=W_{-3} \subset W_{-2}=\Z(1) \subset W_{-1}=H^1(C, \mathfrak{B}; \Z(1)) \subset W_0=H_{\mathfrak{B}, \mathfrak{A}}
\end{equation*} and hence satisfies 
$$
\gr_{-2}^WH_{\mathfrak{B}, \mathfrak{A}}=\Z(1), \quad \gr_{-1}^WH_{\mathfrak{B}, \mathfrak{A}}=H^1(C, \Z(1)), \quad \gr_0^WH_{\mathfrak{B}, \mathfrak{A}}=\Z(0). 
$$

\begin{remark}\
  \begin{enumerate}
  \item By Poincar\'e duality, the biextension $H^1(C \setminus
    \mathfrak{A}, \mathfrak{B}; \Z(1))$ is isomomorphic to the
    biextension $H_1(C \setminus
    \mathfrak{B}, \mathfrak{A}; \Z)$
    which is constructed in the same way, but using homology.
    \smallskip
      \item Going from integral to real coefficients,
    the same construction yields a \textit{real biextension} which
    will be denoted by 
    $$
    H_1(C \setminus \mathfrak{B}, \mathfrak{A}; \R). 
    $$
    It has graded quotients $\R(1)$, $H_1(C, \R)$ and $\R(0)$.
  \end{enumerate}
\end{remark}

\smallskip

\begin{lem}\label{lem:R} The set of isomorphism classes of real
  biextensions with graded quotients $\R(0), H_1(C, \R)$ and $\R(1)$
  is canonically isomorphic to $\R=\C/\R(1)$. Moreover, if we denote
  by $\eta $ the composition of the change of coefficients from $\Z$
  to $\R$ with this isomorphism, then for every pair $\oA$,
  $\oB$ of integer-valued degree zero divisors on $C$ with
  disjoint support the following equality holds
  \begin{displaymath}
    \langle \oA,\oB\rangle = \eta (H_{\oB,\oA}).  
  \end{displaymath}
\end{lem}
\begin{proof}
  The first statement is \cite[Cor. 3.2.9]{Hain} and the second
  is \cite[Prop. 3.3.7]{Hain}. Note that, in this reference, the
  height pairing is defined as the class of the biextension while we
  have defined it as an integral. The content of 
  the \cite[Prop. 3.3.7]{Hain} is that both definitions agree. 
\end{proof}

\subsection{The Poincar\'e bundle} For what follows, it will be more
convenient to reformulate the height pairing in terms of Poincar\'e
bundles. We first recall the construction for a single compact complex
torus $T=V/\Lambda$, where $V$ is a finite dimensional $\C$-vector
space and $\Lambda \subset V$ a cocompact lattice.
By definition, the dual torus $\widehat{T}$ is the quotient $\widehat
T=\widehat{V}/\widehat{\Lambda}$ of the $\C$-vector space
$\widehat{V}= \Hom_{\overline{\C}}(V,\C)$ of $\C$-antilinear
functionals on $V$ by the dual lattice $\widehat{\Lambda}= \{\phi \in
\widehat{V}\ |\ \mathrm{Im}(\phi(\Lambda)) \subset \Z\}$. Observe that
a functional $\phi \in \widehat V$ is uniquely  
determined by its imaginary part $\eta = \mathrm{Im}(\phi) \colon V \rightarrow \R$, thanks to the formula $\phi(v) = \eta(-iv) + i \eta(v)$. 

\smallskip

For $\phi \in \widehat T$, denote by $L_\phi$ the $\C^\times$-bundle on $T$ associated to the  representation of the fundamental group
\begin{equation}\label{cocP}
\pi_1(T)=\Lambda \subset V \xrightarrow{\mathrm{Im}(\phi)} \C \xrightarrow{\exp(2\pi i\cdot)} \C^\times.
\end{equation}

\smallskip

A Poincar\'e bundle $\sP^\times$ is a $\C^\times$-bundle on $T\times
\widehat{T}$, which is uniquely characterized up to isomorphisms by
the following two properties:
\begin{itemize}
 \item[$(i)$] The restriction $\sP^\times|_{\{0\}\times \widehat{T}}$ is trivial. %with a given trivialization.
 \item[$(ii)$] The restriction $\sP^\times|_{T\times \{\phi\}}$ is $L_\phi$.
\end{itemize} Moreover, if $\sP^\times_1$ and $\sP^\times_2$ are two
$\C^\times$-bundles satisfying conditions (i) and (ii) and we choose
trivializations $\sP^\times_i|_{\{(0,0)\}} \simeq \C^{\times}$, then
there is a unique isomorphism $\sP_1^\times \simeq \sP_2^\times$
compatible with the trivializations. A Poincar\'e bundle $\sP^\times$
together with a trivialization $\sP^\times|_{\{(0,0)\}} \simeq
\C^{\times}$ is called a \textit{rigidified} Poincar\'e bundle.

\smallskip

More generally, if $T\to X$ is a holomorphic family of principally
polarized abelian varieties, the dual abelian varieties fit together into a holomorphic family $\widehat
T\to X$. A
Poincar\'e bundle on the product $\pi \colon
T \times_X \widehat T\to X$ is
a $\C^{\times}$-bundle $\sP^{\times}$ such that 
\begin{itemize}
\item[(i)] The restriction of
$\sP^{\times}$ to each fibre is a Poincar\'e bundle. 
\item[(ii)] The restriction to the zero section $s_{0}\colon X\to
T \times_X \widehat T$ is trivial. 
\end{itemize}
A rigidification of $\sP^{\times} $ is an isomorphism
\begin{displaymath}
  s_{0}^{\ast}\sP^{\times}\simeq \sO^{\times}_{X}.
\end{displaymath}

\smallskip
 To extend the Poincar\'e bundle to the space $\sA_g$ of all
 principally polarized abelian varieties of dimension $g$, first
 recall the construction of $\sA_g$. The Siegel domain is by
 definition  
\begin{equation*}
\H_g = \{\text{ $g\times g$ complex symmetric matrix } \Omega\,\,|\
\rm{Im}(\Omega)>0\,\}.
\end{equation*}
The group $\mathrm{Sp}_{2g}(\R)$ acts on $\H_g$ by 
$$\begin{pmatrix}A & B \\ C & D\end{pmatrix}\Omega = (A\Omega+B)(C\Omega+D)^{-1}. 
$$
The quotient $\sA_g= \mathrm{Sp}_{2g}(\Z)\,\backslash\, \H_g$ is the
Siegel moduli space parametrizing principally polarized abelian
varieties of dimension $g$. As a complex manifold, this quotient is
not smooth due to the existence of elliptic fixed points, but it is a
smooth Deligne-Mumford stack and as such is the fine moduli space of
principally polarized abelian varieties.

\smallskip

Denote by $\mathrm{Row}_g(\C)\simeq \C^g$ and $\mathrm{Col}_g(\C)\simeq \C^g$ the $g$-dimensional vector space of row and column matrices, and 
let \begin{equation*}
\widetilde{X}= \H_g \times \mathrm{Row}_g(\C) \times \mathrm{Col}_g(\C) \times \C.
\end{equation*}
Define the group $\widetilde G$ by  
 \begin{align*}
  \widetilde{G} = \Big\{\begin{pmatrix} 1 & \lambda_1 & \lambda_2 & \alpha \\0 & A & B & \mu_1 \\ 0 & C & D & \mu_2 \\ 0 & 0 & 0 & 1\end{pmatrix}\,\,\Big |\,\, \lambda_i \in \text{Row}_g(\R),\ \mu_j \in 
\text{Col}_g(\R), \alpha \in \C, \nonumber \\ \begin{pmatrix}A & B \\ C & D\end{pmatrix}\in \mathrm{Sp}_{2g}(\R)\Big\}.
 \end{align*}

The space $\widetilde X$ is a homogeneous space for the group $\widetilde G$ with respect to the action given by 
\ml{action:Omega}{\begin{pmatrix} 1 & 0 & 0 & 0 \\0 & A & B & 0 \\ 0 & C & D & 0 \\ 0 & 0 & 0 & 1\end{pmatrix}(\Omega,W,Z,\rho) =((A\Omega+B)(C\Omega+D)^{-1}, \\ W(C\Omega+D)^{-1},\prescript{\mathrm t}{}{(C\Omega+D)}^{-1}Z,\rho-W\prescript{\mathrm t}{}{C}\prescript{\mathrm t}{}{(C\Omega+D)}^{-1}Z),}
\begin{equation}
\label{W}
\begin{pmatrix} 1 & \lambda_1 & \lambda_2 & 0 \\0 & \id_g & 0 & 0 \\ 0 & 0 & \id_g & 0 \\ 0 & 0 & 0 & 1\end{pmatrix}(\Omega,W,Z,\rho) = 
(\Omega, W+\lambda_1\Omega+\lambda_2,Z,\rho+\lambda_1Z), 
\end{equation}
\begin{equation}
\label{action:Z}
\begin{pmatrix} 1 & 0 & 0 & 0 \\0 & \id_g & 0 & \mu_1 \\ 0 & 0 & \id_g & \mu_2 \\ 0 & 0 & 0 & 1\end{pmatrix}(\Omega,W,Z,\rho) = 
(\Omega,W,Z+\mu_1-\Omega\mu_2,\rho-W\mu_2),
\end{equation}
\begin{equation}
\label{action:alpha}\begin{pmatrix} 1 & 0 & 0 & \alpha \\0 & \id_g & 0 &0 \\ 0 & 0 & \id_g &0 \\ 0 & 0 & 0 & 1\end{pmatrix}(\Omega,W,Z,\rho) = 
(\Omega,W,Z,\rho+\alpha). 
\end{equation} 

\smallskip

Denote by $\widetilde{G}(\Z)\subset \widetilde{G}$ the subgroup consisting of those matrices with entries in $\Z$. The matrices in \eqref{action:alpha} form a normal subgroup $N$ of $\widetilde{G}$, so we can take the quotient $G=\widetilde{G}/N$ and consider $G(\Z)=\widetilde{G}(\Z)/N(\Z).$ The following result gives a characterization of the Poincar\'e bundle.

\begin{thm} \begin{enumerate}
\item[]
\smallskip
\item The quotient
\begin{equation*}
\sE_{g}= G(\Z)\,\backslash\,\Bigl(\,\H_g\times \mathrm{Row}_g(\C)\times \mathrm{Col}_g(\C)\,\Bigr)
\end{equation*}
is isomorphic to the universal family of abelian varieties and their
duals over the fine moduli stack $\sA_g$. 
\smallskip
\item Under the previous isomorphism, the quotient
\begin{equation}
 \label{invP}
 \sP_{g}^{\times}=\widetilde{G}(\Z)\,\backslash\,\Bigl(\,\H_g\times
 \mathrm{Row}_g(\C)\times \mathrm{Col}_g(\C)\times \C\,\Bigr) 
\end{equation}
is a Poincar\'e bundle over $\sE_{g}$. Moreover, there is a canonical
isomorphism
\begin{displaymath}
  (\mathrm{Sp}_{2g}(\Z)\times N(\Z))\,\backslash\,\Bigl(\,\H_g\times
 \{(0,0)\}\times \C\,\Bigr)=\sA_{g}\times \C^{\times}. 
\end{displaymath}
that rigidifies $\sP^{\times}_{g}$.
\end{enumerate}
\end{thm}
\begin{proof}
  This result is classical. See for instance \cite[\S 8.7]{AbV} for
  the construction of the universal family of abelian varieties.
  We start by sketching the construction of the isomorphism claimed in the first
  statement. Let $T$ be an element of $\sA_g$ which is the image of
  the element $\Omega \in \H_{g}$. Denote by $\omega_1, \dots,
  \omega_g$ the rows of $\Omega$, and let $e_1,\dots, e_g $ be the
  standard basis of $\C^g$, so that the lattice 
$\Lambda_{\Omega }$ is generated by $\omega_k, e_j$. Then the
corresponding abelian variety is $T=\C^g/\Lambda_{\Omega }$. 
  We
identify $\C^{g}$ with $\mathrm{Row}_g(\C)$ and $\Lambda $ with
$\mathrm{Row}(\mathbb Z^g)\oplus \mathrm{Row}(\mathbb Z^g)$ using the
above basis, so the inclusion $\Lambda \inj \C^{g}$ is given by
$(\lambda_1, \lambda_2)\mapsto \lambda_1 \Omega + \lambda_2$. When we
want to distinguish between an abstract vector $v\in \C^{g}$ and its
image in $\mathrm{Row}_g(\C)$ we will denote the latter by
$W_{v}$. 
   
By \eqref{W}, the action of $\widetilde G$ identifies $W\in
\mathrm{Row}_g(\C)$ with $W + \lambda_1 \Omega + \lambda_2$, for
$\lambda_1, \lambda_2 \in \mathrm{Row}(\mathbb Z^g)$. Thus, the image
of $W$ in the 
quotient 
$\widetilde G\backslash \widetilde X$ varies in $T= \C^g/\Lambda$.

The action of $\widetilde G$ identifies $Z \in \mathrm{Col}(\C^g)$
with  $Z+\mu_1-\Omega\mu_2$.
We verify as follows that the class of $Z$ in $\widetilde G \backslash
\widetilde X$ varies in the dual $\widehat T$ of $T$. If $\eta$ denotes, as before, the imaginary part of $\phi \in \widehat{\C^g}$, we have
$\phi(v) = \eta(-iv) + i \eta(v)$ for all $v$.
First, we identify $\widehat{\C^g} = \Hom_{\overline \C}(\C^g,
\C)$ with $\mathrm{Col}(\C^g)$, via the identification $\phi \in
\widehat{\C^g} \mapsto Z_\phi = \mu_1 - \Omega \mu_2,$ where $\mu_1 =
\mathrm{Col}\bigl(\eta(\omega_1), \dots, \eta(\omega_g)\bigr)$ and
$\mu_2 = \mathrm{Col}\bigl(\eta(e_1), \dots, \eta(e_g)\bigr)$. Under
this identification, the pairing between $\phi $ and $v$ is given by 
\begin{equation*}
\phi (v)=-\overline{W_{v}}\mathrm{Im}(\Omega)^{-1}Z_{\phi },
\end{equation*}
while its imaginary part is 
\begin{equation*}
\eta(v)=\mathrm{Im}(\phi (v))=\lambda _{1}\mu _{1}+\lambda _{2}\mu
_{2}.
\end{equation*}
Therefore, an element $\phi _{0}\in \widehat {\C^{g}}$ belongs to
$\widehat \Lambda$  if and only if $\eta_0 = \mathrm{Im}(\phi_0) \in
\Hom_{\Z}(\Lambda, \Z)$, which amounts to ask that the two associated vectors  
 $\mu_{1}$ and $\mu_2$ have integer coordinates. 
 We get $\widehat{\C^g} /\widehat \Lambda \simeq \C^g/\{\mu_1
 - \Omega \mu_2 \,|\, \mu_1, \mu_2 \in \Z^g\}$. In this way, we have verified
 that the fibre of $\sE_{g}\to \sA_{g}$ over a point $T\in \sA_{g}$ is
 identified with $T\times \widehat T$. This
 identification can be extended to an isomorphism of $\sE_{g}$ with
 the universal family of abelian varieties and their duals. 
 
 \smallskip
 
 Denote by $\sL$ the $\C^{\times}$ bundle obtained as the quotient
 \eqref{invP}. 
 If we restrict the actions \eqref{action:Omega}  and \eqref{action:Z}
 to the points
 of the form $(\Omega,0,Z,\rho)$ we see that the restriction of $\sL$
 to the set $W=0$ (which is the zero section in the abelian variety)
 is trivial. Thus we obtain the first condition that
 characterizes the Poincar\'e bundle. We now fix $\Omega_0\in \H_g$ and
 $Z_0\in\mathrm{Col}(\C^g)$. Denote $X_{0}$ the subvariety of
 $\widetilde X$ of equations $\Omega =\Omega _{0}$, $Z=Z_{0}$ and
 $\phi _{0}\in \widehat{\C^{g}}$ corresponding to $Z_{0}$. That is,
 \begin{equation*}
   \phi_{0} (v)=-\overline{W_{v}}\mathrm{Im}(\Omega_{0})^{-1}Z_{0}.
 \end{equation*}
 The
 restriction of the action 
 \eqref{W} to $X_{0}$ reads  
 \begin{displaymath}
 (\lambda_1,\lambda_2)(\Omega_0,W,Z_0,\rho)=
 (\Omega_0,W+\lambda_1\Omega_0+\lambda_2,\rho+\lambda_1Z_0).
 \end{displaymath}
 Hence the restriction of $\sL$ to the abelian variety covered by $X_{0}$
 is the $\C^{\times}$ bundle determined by the cocycle
 $$
 a((\lambda_1,\lambda_2),W)=\exp(2\pi i \lambda_1Z_0)
 $$
 Consider the holomorphic function $\psi\colon X_{0}\to \C$ given by
 $\psi (W)=\exp(2\pi i W\mathrm{Im}(\Omega _{0})^{-1}\mathrm{Im}(Z_0))$.
 The cocycle
 \begin{displaymath}
   b((\lambda_1,\lambda_2),W)=a((\lambda_1,\lambda_2),W)\psi
   (W+\lambda_1\Omega_0+\lambda_2)^{-1} \psi (W)
 \end{displaymath}
is equivalent to $a$ and hence defines an isomorphic bundle. Computing
this cocycle we obtain
\begin{displaymath}
  b((\lambda_1,\lambda_2),W)=\exp(2\pi i \mathrm{Im}(\phi_{0}(\lambda
  _{1}\Omega +\lambda _{2})).
\end{displaymath}
By \eqref{cocP} this cocycle determines the line bundle $L_{\phi
  _{0}}$, so $\sL$ satisfies also the second condition that
determines the Poincar\'e bundle. In consequence, we have seen that the
restriction of $\sP^{\times}_{g}$ to each fibre of $\sE_{g}\to
\sA_{g}$ is a Poincar\'e bundle. The stated rigidification shows in
particular that the restriction to the zero section is trivial
implying the statement.
\end{proof}

\begin{thm} \label{thm:2} Let $\sP^{\times}$ be a rigidified
  Poincar\'e bundle. Then
  there is a unique metric on $\sP^{\times}$ whose curvature is
  translation invariant and that, under the rigidification, satisfies
  $\|1\|=1$. Moreover, for the Poincar\'e bundle $\sP_{g}$ over the universal
  family $\sE_{p}$ this metric is given, for an element
  $(\Omega,W,Z,\rho)$ in $\widetilde{X}$, by  
\begin{equation}
  \label{eq:64} 
  \log||(\Omega,W,Z,\rho)|| = 
  \Big(-2\pi\mathrm{Im}(\rho)+
  2\pi\mathrm{Im}(W)
  \bigl(\mathrm{Im}(\Omega)\bigr)^{-1}\mathrm{Im}(Z)\Big). 
\end{equation}
\end{thm}
\begin{proof}
It is well known that the invariance of the curvature form fixes the
metric up to a multiplicative constant on each fibre but the
compatibility with the rigidification fixes this constant.

Consider the space $\widetilde X'=\H_g\times \mathrm{Row}_g(\C)\times
\mathrm{Col}_g(\C)\times \C^{\times}$ and the map $\widetilde X\to
\widetilde X'$
that sends $\rho $ to $s=\exp(2\pi i \rho )$. Then $\widetilde X'$ is
a trivial $\C^{\times}$-bundle over $\H_g\times \mathrm{Row}_g(\C)\times
\mathrm{Col}_g(\C)$. The formula \eqref{eq:64} determines a hermitian
metric on this trivial bundle given by
\begin{equation}\label{metric}
  \|(\Omega,W,Z,s)\|^{2}=|s|^{2}\exp(4\pi\mathrm{Im}(W)
  \bigl(\mathrm{Im}(\Omega)\bigr)^{-1}\mathrm{Im}(Z)).
\end{equation}

Let $\widetilde G_{\R} \subset \widetilde G$ be the subgroup
consisting of matrices with $\alpha \in \R$. The fact that
\eqref{metric}  
induces a metric in the Poincar\'e bundle whose curvature is invariant
under translation follows from the invariance of the function
\eqref{eq:64} under the action of 
$\widetilde G_{\R}$, which is a straightforward verification.  
\end{proof}

\subsection{The Poincar\'e bundle and the archimedean
  height}
\label{sec:biext-arch-heights}

The interest for us on the Poincar\'e bundle is consequence of the
relation between biextensions and the Poincar\'e bundle due to Hain
\cite{Hain}.

\smallskip

Let $H$ be a principally polarized pure Hodge structure of weight $-1$ and type
$\{(-1,0),(0,-1)\}$ and $T_{H}$ the corresponding principally
polarized abelian variety. Let $\sB(H,\Z)$ be the set of isomorphism
classes of
biextensions of $H$. That is, the isomorphism classes of mixed Hodge
structures $E$  of weights $-2$, $-1$ and $0$ with
\begin{displaymath}
  \Gr^{W}_{-2}(E)=\Z(1),\quad \Gr^{W}_{-1}(E)=H,\quad \Gr^{W}_{0}(E)=\Z(0).
\end{displaymath}
There are natural maps
\begin{displaymath}
  \begin{matrix}
    \sB(H,\Z)&\longrightarrow & \Ext^{1}(\Z(0),H)=T_{H}\\
    E&\longmapsto& E/W_{-2}E.
  \end{matrix}
\end{displaymath}
and
\begin{displaymath}
  \begin{matrix}
    \sB(H,\Z)&\longrightarrow & \Ext^{1}(H,\Z(1))=\widehat T_{H}\\
    E&\longmapsto& W_{-1}E.
  \end{matrix}
\end{displaymath}
Thus we obtain a map $\sB(H,\Z)\to T_{H}\times \widehat T_{H}$. 
The canonical isomorphism $\Ext^{1}(\Z(0),\Z(1))=\C^{\times}$ induces a
structure of $\C^{\times}$-bundle on $\sB(H,\Z)$ and a rigidification
of $\sB(H,\Z)$.

\smallskip
 
Let $\sB(H,\R)$ be the set of isomorphism classes of real
biextensions. By Lemma \ref{lem:R} we can identify $\sB(H,\R)$ with
$\R$. We have already denoted by $\eta \colon \sB(H,\Z)\to \R$. 

\begin{thm}[Hain \cite{Hain}]\label{thm:1} The bundle $\sB(H,\Z)$ is
  a rigidified Poincar\'e bundle and the invariant metric is given by
  \begin{displaymath}
    \log \|E \|=\eta(E).
  \end{displaymath}
\end{thm}

Using now the relation between the height pairing and the
biextensions we can relate the height pairing and the Poincar\'e
bundle. Summing up Theorem \ref{thm:1} and Lemma \ref{lem:R} we deduce:

\begin{prop}\label{height-pm} Let $C$ be a smooth projective curve
  over $\C$ and $\oA$,
  $\oB$ integer-valued degree zero divisors on $C$ with disjoint
  support. The following three quantities coincide:
\begin{enumerate}
\item[(a)] $\langle\,\mathfrak A\,,\mathfrak B\,\rangle$,
\item[(b)] $\log ||H_{\mathfrak{B}, \mathfrak{A}}||$,
\item[(c)] $\eta(H_{\mathfrak{B}, \mathfrak{A}})$.
\end{enumerate}
\end{prop}

\begin{ex} Let $C=\P^1$ and consider the divisors $\mathfrak{A}=z_1-z_2$ and $\mathfrak{B}=z_3-z_4$, where $z_i$ are four distinct points of $\P^1$. Then: 
$$
\langle \mathfrak{A}, \mathfrak{B} \rangle=\mathrm{Re} \int_{z_4}^{z_3} \left(\frac{1}{z-z_1}-\frac{1}{z-z_2}\right)dz=\log \big|\frac{(z_3-z_1)(z_4-z_2)}{(z_3-z_2)(z_4-z_1)}\big|. 
$$ Observe that the term inside the absolute value is nothing else than the cross-ratio of the points $z_i$. Since $H^1(\P^1, \Z(1))=0$, the biextension associated to $\mathfrak{A}$ and $\mathfrak{B}$ is in this case simply the extension 
$$
0 \to \Z(1) \to H^1(\P^1 \setminus \{ z_1, z_2 \}, \{ z_3, z_4 \}; \Z(1)) \to \Z(0) \to 0. 
$$
\end{ex}

\section{The asymptotic of the height pairing}\label{sec:nilp}

The goal of this section is to prove Theorem \ref{thm:6} from the introduction, which relates the asymptotic of the height pairing between degree zero divisors with disjoint support, as a family of smooth curves degenerates to a stable curve $C_0$, to the ratio of the first and the second Symanzik polynomials of the dual graph of $C_0$.

\subsection{Asymptotic of the height pairing}

Let $\Delta$ be a small open disc around $0 \in \C$, and write
$\Delta^*=\Delta \setminus \{0\}$ and $S=\Delta^{3g-3}$. Consider the
versal analytic deformation $\pi \colon \sC \to S$ from Section
\ref{riemannsurf} of the stable curve $C_0$. The stability of $C_0$
implies in particular that we deal with curves of genus $g\ge 2$ and
that the dual graph $C_0$ has $\le 3g-3$ edges. These assumptions can
be removed if one works systematically with moduli of stable marked
curves.

Recall that the fibres are smooth outside a normal crossing divisor
$D=\bigcup_{e \in E} D_e \subset S$, with irreducible components
indexed by the set of singular points of $C_0$. We denote by $U$ the
complement of $D$ in $S$ 
and we identify it with $U=(\Delta^*)^{E}\times
\Delta^{3g-3-|E|}$. The universal cover is then 
\begin{equation}\label{eq:universal-cover}
\widetilde{U} = \mathbb H^E \times \Delta^{3g-3-|E|} \longrightarrow U, 
\end{equation} where the map is induced by $z_e \mapsto \exp(2\pi i
z_e)$ in the first factors and identity on the second factors.

\smallskip

We assume moreover that  we are given two collections
$$
\sigma_1=\{ \sigma_{\rl, 1}\}_{\rl=1, \ldots, n}, \quad \sigma_2=\{\sigma_{\rl, 2}\}_{\rl=1, \ldots, n}
$$
of sections $\sigma_{\rl, i} \colon S \to \mathcal{C}$ of $\pi$. Since $\sC$ is regular over $\C$, the sections cannot pass through double points of $C_0$. Thus, for each $\rl$, 
$\sigma_{\rl,i}(S)\cap C_0$ lies in a unique irreducible component $X_{v_\rl}$ of $C_0$, corresponding to a vertex $v_\rl$ of $G$. We assume further that the sections $\sigma_{\rl,1}$ and $\sigma_{\rl,2}$ are distinct on $C_0$. It
follows, possibly after shrinking $S$, that ${\sigma_1}$ and $\sigma_2$ are disjoint as well.

\smallskip

 Let $\underline \p_1 = \{\p_{\rl,1}\}_{\rl=1}^n \in (\R^D)^{n,0}$ and 
 $\underline \p_2 = \{\p_{\rl,2}\}_{\rl=1}^{n}\in (\R^D)^{n,0}$ be two collections of external momenta satisfying 
 the conservation law. 
 We label the marked points $\sigma_{\rl,i}$ with $\p_{\rl,i} \in \R^{D}$, and we write $\underline \p_1^G=(\p_{v, 1}^G)$ and $\p_2^G=(\p_{v, 2}^G)$ for the \textit{restriction} of $\underline \p_1$ and $\underline \p_2$ to $G$. By definition, for each vertex 
$v$ of $G$, the vector  $\p_{v,i}^G$ is the sum of all $\p_{\rl,i}$ with $v_\rl = v$. 
\smallskip

For any $s \in S$, let $\mathfrak A_s $ and $\mathfrak B_{s}$ denote the $\R^{D}$-valued degree zero
divisors on $C_s$ 
$$
\mathfrak A_s = \sum_{\rl=1}^n \ps_{\rl,1}\sigma_{\rl,1}(s), \qquad \mathfrak B_s = \sum_{\rl=1}^n \p_{\rl,2}\sigma_{\rl,2}(s).
$$ Recall that in Section
 \ref{sec:archimedean-heights} we have extended the usual archimedean height pairing to $\R^{D}$-valued degree zero divisors by means of the given Minkowski bilinear form.  
We thus get a function 
\begin{equation*}
U \longrightarrow \R, \quad s \longmapsto \langle \mathfrak{A}_s, \mathfrak{B}_s \rangle. 
\end{equation*}

\begin{defn} An \textit{admissible segment in $S$} is a continuous map
  $$
  \underline t: I=(0,\varepsilon) \rightarrow U=(\Delta^*)^{E}\times
  \Delta^{3g-3-|E|}
  $$ from an open interval of positive length $\varepsilon$, which
  satisfies the following:
    \begin{enumerate}
   \item[(i)] letting $(t_e)_{e \in E}$ denote the coordinates in the factor $(\Delta^\ast)^E$, the limit $\lim_{\alpha'\rightarrow 0}
     |t_e(\alpha')|^{\alpha'}$ exists and belongs to $(0, 1)$ for all $e \in E$; 
     
     \vspace{1mm}
    
    \item[(ii)] the segment can be extended to a continuous map $\underline t : [0,\varepsilon)
      \rightarrow S$. 
  \end{enumerate}
\end{defn}
Note that it follows from property (i) that $t_e(0) =0$ for all
$e\in E$. 

\begin{ex} \
\begin{enumerate}
\item Given tuples of real numbers $(x_e)_{e \in E}$ and $(Y_e)_{e \in E}$ with $Y_e>0$ for all $e \in E$, we define a map $\underline{z} \colon (0, 1) \to \mathbb{H}^E \times \{0\} \subset \widetilde U$ by $z_e(\alpha')=x_e+i \frac{Y_e}{2\pi \alpha'}$. Projecting to $U$ by the universal cover \eqref{eq:universal-cover}, we get an admissible segment $\underline{t} \colon (0, 1) \to U$ for which $|t_e(\alpha')|^{\alpha'}=\exp(-Y_e)$. \newline\noindent

\item Let $\Delta_\delta^*$ be a punctured disc of radius $\delta$ centered at the origin, and $\gamma: \Delta_\delta^* \to U$ be an analytic map which can be analytically extended to $\widetilde{\gamma}: \Delta_\delta \to S$. Let $\epsilon = - \frac{1}{\log\delta}$ and let $\pi: \Delta_\delta^* \to (0,\epsilon)$ be the map $\pi(t) = - \frac{1}{\log |t|}$. Then for any continuous section $\eta$ of $\pi$, the composition $\gamma\circ\eta: (0,\epsilon) \to U$ is an admissible segment. 
\end{enumerate}
\end{ex}

\smallskip

Our first result describes the asymptotic behavior of the height pairing $\langle \mathfrak A_s,
  \mathfrak B_s \rangle$ as the smooth curves $C_s$ degenerate to $C_0$ through an admissible segment. 

\begin{thm}\label{thm:main} For any admissible segment $\underline t \colon I \rightarrow U$ the following asymptotic estimate holds
\begin{equation}\label{asshp}
\lim_{\alpha'\to 0} \alpha'\langle \mathfrak A_{\underline t(\alpha
    ')},
  \mathfrak B_{\underline t(\alpha ')}\rangle =\frac{\phi_G(\underline \p_1^G,
  \underline \p_2^G,\underline Y)}{\psi_G(\underline Y)
}, 
\end{equation}
where, for each edge $e\in E$, we define
$$
Y_e = - \lim_{\alpha'\to 0}
\log|t_e(\alpha')|^{\alpha'}>0,
$$ and $\psi_G$ and $\phi_G$ denote the first
and second Symanzik polynomials of $G$.
\end{thm}

\smallskip

 Before proving the theorem, we need to consider the period map
 obtained from the variation of the biextension mixed Hodge structures given by the divisors $\mathfrak A_s$ and $\mathfrak B_s$. This is what we do next.

\subsection{The period map and its monodromy}

Throughout this section we assume that the divisors $\mathfrak A_s$ and $\mathfrak
B_s$ are integer-valued, that is, $\ps_{\rl,i}\in \Z$ for all $\rl,i$. We see the $\underline \ps_{i}$ as row vectors. 

\smallskip

The family of mixed Hodge structures $H_{\oB_{s},\oA_{s}}$ fit together into an admissible variation of mixed Hodge structures (see
\cite{SZ:vmh} for the definition). This can be seen as follows. Using
the theory of mixed Hodge modules \cite{Saito1}, \cite{Saito2} one can
form a mixed Hodge module $H_{\oB_{s},\oA_{s}}$. Since the relative
homology $H_{1}(C_{s}\setminus \oA_{s},\oB_{s};\Z)$ is a local system,
then $H_{\oB_{s},\oA_{s}}$ is an admissible variation of mixed Hodge
structures. See \cite{Asakura} for a survey of mixed Hodge modules
with all the needed properties.

\smallskip

In what follows we give a description of the period map of the
variation of $H_{\oB_{s},\oA_{s}}$ and its monodromy. Since the period
map is only well 
defined up to the action of the group $\widetilde{G}(\Z)$ we have to
make some choices. We start by stating explicitely all the choices.
We fix basepoints $s_0 \in 
U$ and $\tilde
s_0 \in \widetilde U$ lying above $s_0$, and a symplectic  
basis 
\begin{equation*}
a_1,\dotsc,a_{g},b_1,\dotsc,b_{g} \in H_1(C_{s_0},\Z) =
A_{0}\oplus B_{0}.
\end{equation*} 
such that the space of vanishing cycles $A$ is generated by
$a_1,\dots, a_h\in A$, and $b_1, \dots, b_h$ generate $H_1(C_{s_0},
\Z)/A' \simeq H_1(G, \Z)$ as in \eqref{eq:aux}.

\smallskip

For $i=1,2$, we write $\Sigma _{i,s}=\{\sigma _{1,i}(s),\dots ,\sigma
_{n,i}(s)\}$, $\Sigma_{s} =\Sigma _{1, s}\cup \Sigma _{2, s}$ and
$\Sigma _{i}=\bigcup_{s}\Sigma _{i,s}$. We 
lift the classes $a_{j}$ and $b_{j}$, $j=1,\dots ,g$ to elements of
$H_{1}(C_{s_{0}}\setminus \Sigma _{s_{0}},\Z)$ by choosing loops that
do not meet the points in $\Sigma _{s_{0}}$. We will denote these new classes also by $a_{j}$ and $b_{j}$.

\smallskip

Since the cohomology groups 
$H_{1}(C_{s}\setminus \Sigma _{s},\Z)$ form a local system, we can
spread out this symplectic basis to a basis 
\begin{equation*}
a_{1,\tilde s}, \dots, a_{g,\tilde s}, b_{1,\tilde s}, \dots,
b_{g,\tilde s} 
\end{equation*}
of $H_{1}(C_{s}\setminus \Sigma _{s},\Z)$, for any $s\in U$ and $\tilde s\in
\widetilde U$ over it. 
If there is no risk of confusion, we drop $\tilde s$, and simply use
$a_i$ and $b_i$ for these elements. Note also that, 
since $A_{0}$ is isotropic and contains the subspace of vanishing
cycles, the Picard-Lefschetz formula \eqref{eq:42} implies that
the elements $a_{i,\tilde s}$ only depend on $s$ and not on $\tilde
s$. Thus we will also denote them by $a_{i,s}$.

\smallskip

By the admissibility of $H_{\oB,\oA}$ we know that
$H_{\oB,\oA}\otimes_{\C}\sO_{U}$ can be extended to a
holomorphic vector bundle over $S$ and that  
$F^{0}W_{-1}H_{\oB,\oA}$ can be extended to a coherent subsheaf of it.
From this we deduce the existence of a collection of $1$-forms
$\{\omega_{i}\}_{i=1, \ldots, 
  g}$ on $\pi ^{-1}(U)\subset \sC$ such that, for each $s\in
U$, the forms $\{\omega_{i, s}\coloneqq
\omega_{i}|_{C_{s}}\}_{i=1, \ldots, g}$ 
are a  basis  of the
holomorphic differentials on $C_s$ and
\begin{equation} \label{eq:7}
 \int_{a_{i, s}}
\omega_{j, s}=\delta_{i, j}. 
\end{equation}
Then the classical period matrix for the
family of curves $\sC$ is
$(\,\int_{b_{i, s}}\omega_{j,s}\,)$. 

\smallskip

We choose an integer valued 1-chain
$\gamma_{\oB_{s_0}}$ on $C_{s_0} \setminus \Sigma _{1,s_{0}}$
having $\oB_{s_{0}}$ as boundary. By adding a linear combination of
the $b_{j}$ if needed, we can assume that
\begin{equation}
  \label{eq:5}
  \langle a_{i},\gamma
_{\oB_{s_{0}}}\rangle=0.
\end{equation}
The chain $\gamma_{\oB_{s_0}}$ determines a class 
$$
[\gamma_{\oB_{s_0}}]\in H_{1}(C_{s_0} \setminus
\Sigma _{1,s_{0}},\Sigma _{2,s_{0}},\Z)
$$ 
that we can spread to classes $\gamma_{\oB_{\tilde s}}$ as before. 

\smallskip

Invoking again the admissibility of $H_{\oB,\oA}$, we can find 
a 1-form $\omega_{\oA}$ on $\pi
^{-1}(U)\setminus \Sigma _{1}$ such that each restriction
$\omega _{\oA,s}\coloneqq \omega_{\oA}|_{C_{s}} $ is a holomorphic
form of the third kind, with residue $\oA_s$ and normalized in such a
way that
\begin{equation}\label{eq:8}
  \int _{a_{i},s}\omega _{\oA,s}=0,\quad i=1,\dots,g.
\end{equation}
Note that this last condition is easily achieved by adding to
$\omega_{\oA}$ a suitable linear combination of the $\omega_{i}$.   

\begin{prop} The \textit{period map} of the variation of mixed Hodge
  structures $H_{\mathfrak{B}_s, \mathfrak{A}_s}$ is given by
\begin{align}
\widetilde\Phi \colon \widetilde U &\longrightarrow \mathbb
H_{g}\times \mathrm{Row}_{g}(\C) \times \mathrm{Col}_{g}(\C) \times \C \nonumber \\
\tilde{s} &\longmapsto \Bigl(
\,\bigl(\,\int_{b_{i, \tilde s}}\omega_{j,s}\,\bigr)_{i,j}\,,
\, \bigl(\,\int_{\gamma_{\mathfrak B, \tilde s}}\omega_{j,s}\,\bigr)_j\,,
\,\bigl(\,\int_{b_{i,\tilde s}}\omega_{\mathfrak A,s}\,\bigr)_i\,,
\,\int_{\gamma_{\mathfrak B, \tilde s}}\omega_{\mathfrak A,s}\,\Bigr)
\label{eq:51}.
\end{align}
\end{prop}

\begin{proof} We drop the index $s$ and work pointwise.    
Recall the definition of the biextension mixed Hodge structure
$H_{\mathfrak B, \mathfrak A}$ from Section
\ref{sec:biext}. 
The integral part $H_{\mathfrak B, \mathfrak A}$ has a basis given by
$\alpha_{\oA} ,a_{1},\dots,a_{g},b_{1},\dots,b_{g},\gamma
_{\oB}$, where $\alpha _{\oA}$ is the generator of $\Q(1)\subset
W_{-2}H_{\mathfrak B, \mathfrak A,\Q}$  determined by the divisor $\oA$. This
means that, if $\delta _{\rl} $ is a small negatively oriented disc
centered at $\sigma_{\rl,1}$, then the image of  $\partial \delta
_{\rl} $ in $H_{\mathfrak B, \mathfrak A,\Q}$ is $\ps_{\rl,1} \alpha _{\oA}$.

 The
quotient $H_{\mathfrak B, \mathfrak A,\C}/F^{0}$ has a basis given by the classes
\begin{equation}
  \label{eq:12}
  [\alpha_{\oA}] ,[a_{1}],\dots,[a_{g}]
\end{equation}
The class of the
biextension $H_{\mathfrak B, \mathfrak A}$ is given by the expression of 
the classes $[b_{1}],\dots,[b_{g}],[\gamma_{\oB}]$ is the basis
\eqref{eq:12}: 
\begin{displaymath}
  \begin{pmatrix}
    W & \rho \\ \Omega  & Z
  \end{pmatrix}
\end{displaymath}
with $\Omega \in \H_{g}$, $W\in \mathrm{Row}_{g}(\C)$, $Z\in
\mathrm{Col}_{g}(\C)$ and $\rho \in \C$. Given a path $\gamma $
representing a class in $H_{1}(C\setminus \oB,\oA;\Z)$, then, by the choice
of the forms $\omega  _{j}$ and $\omega _{\oA}$, the
expression of the class $[\gamma ]$ is the basis \eqref{eq:12} is given by
\begin{displaymath}
   \Big( \int_{\gamma }\omega _{\oA}, \int_{\gamma }\omega
   _{1},\dots,\int_{\gamma }\omega_{g} \Big),
\end{displaymath}
which proves the proposition. 
\end{proof}

We next describe the action of  the logarithm of monodromy maps $N_e$
on the entries in \eqref{eq:51}, for $e\in E$.  To this end we observe
first that, since the 
sections $\sigma _{l,i}$ do not meet the double points of $C_{0}$, the
vanishing cycles $a_{e}\in H_{1}(C_{s_{0}},\Z)$ can be lifted
canonically to cycles in $H_{1}(C_{s_{0}}\setminus \Sigma
_{s_{0}},\Z)$. These cycles will also be denoted by $a_{e}$. In this
group we can write 
\begin{equation} \label{eq:2}
  a_{e}=\sum_{i}c_{e,i} a_{i}+\sum _{l} d_{e,l,1} \gamma _{l,1} +
  \sum _{l} d_{e,l,2} \gamma _{l,2},
\end{equation}
where $\gamma _{l,i}$ is a small negatively oriented loop around
$\sigma _{l,i}(s_{0})$. By the choice of the basis $\{a_{i},b_{i}\}$
the coefficients  $c_{e,i}$ are zero for $i>h$. 

\smallskip

By the Picard-Lefschetz formula \eqref{eq:42}, the assumption
\eqref{eq:5} and the formula \eqref{eq:2}, we deduce that
\begin{align}
  N_{e}(b_{i})
  &=-\langle b_{i},a_{e}\rangle a_{e} = c_{e,i}a_{e},\label{eq:10}\\
  N_{e}(\gamma _{\oB_{s_{0}}})
  &=-\langle \gamma_{\oB_{s_{0}}},a_{e}\rangle a_{e} =
    -a_{e}\sum_{l}\ps_{l,2}d_{e,l,2}.\label{eq:11}
\end{align}
Since the forms $\omega _{j}$ and $\omega _{\oA}$ are defined
globally, they are invariant under monodromy. The integral of these
forms with respect to the vanishing cycles is computed using 
\eqref{eq:2}, \eqref{eq:8} and \eqref{eq:7}:
\begin{equation}\label{eq:9}
  \int_{a_{e}}\omega _{j}=c_{e,j},\quad \int_{a_{e}}\omega
  _{\oA_{s_{0}}}=\sum_{l} \ps_{l,1}d_{e,l,1}.
\end{equation}

Applying \eqref{eq:10}, \eqref{eq:11} and \eqref{eq:9} we deduce 
\begin{align*}
  N_e(\int_{b_i}\omega_{j,s_{0}})
  &= -\langle b_i,a_e\rangle\int_{a_e}\omega_{j,s} =
    c_{e,i}c_{e,j}\\
  N_e(\int_{\gamma_{\oB_{s_{0}}}}\omega_{j,s_{0}})  
  &= -\langle \gamma_{\oB_{s_{0}}},a_e\rangle\int_{a_e}\omega_{j,s}=
    -c_{e,j}\sum_{l}\ps_{l,2}d_{e,l,2}\\ 
  N_e(\int_{b_i}\omega _{\oA_{s_{0}}})
  &=-\langle b_i,a_e\rangle\int_{a_e}\omega_{\oA_{s_{0}}}=
    c_{e,i}\sum_{l}\ps_{l,1}d_{e,l,1},\\ 
  N_e(\int_{\gamma_{\oB_{s_{0}}}}\omega _{\oA_{s_{0}}})
  &= -\langle \gamma_{\oB_{s_{0}}},a_e\rangle
    \int_{a_e}\omega_{\oA_{s_{0}}}=
    -\Big(\sum_{l}\ps_{l,1}d_{e,l,1}\Big)
    \Big(\sum_{k}\ps_{k,2}d_{e,k,2}\Big).
  \end{align*}
We introduce the matrices $\widetilde M_{e}$, $\widetilde W_{e}$,
$\widetilde Z_{e}$ and
$\Gamma _{e}$ given by
\begin{gather*}
  (\widetilde M_{e})_{i,j}=c_{e,i}c_{e,j},\qquad
  (\widetilde W_{e})_{l,j}=-c_{e,j}d_{e,l,2},\\ (\widetilde
  Z_{e})_{i,l}=c_{e,i}d_{e,l,1},\qquad  
  (\Gamma _{k,l})=-d_{e,k,2}d_{e,l,1}.
\end{gather*}
Then the logarithm of the monodromy is given by the element of the Lie
algebra of $\widetilde G$
\begin{equation}\label{eq:14}
N_{e}=
  \begin{pmatrix} 0 & 0 & \underline{\ps}_{2}\widetilde W_{e} &
    \underline{\ps}_{2}\Gamma _{e} {^{t}\underline{\ps}_{1}}  \\0 &
    0 & \widetilde M_{e} & \widetilde Z_{e} {^{t}\underline{\ps}_{1}}\\
    0 & 0 & 0 & 0 \\ 0 & 0 & 0 & 0\end{pmatrix}. 
\end{equation}
Note that all the entries of this matrix are integers. 

\smallskip

By Proposition~\ref{rmk25}, the matrix $\widetilde M_{e}$ is the
$h\times h$ matrix $M_{e}$ from Section \ref{sec:sym}  
filled with zeros to a $g \times g$
matrix. Similarly, the matrix $\widetilde W_{e}$ (resp. $\widetilde{Z}_{e}$) is the extension with
zeros of a matrix $W_{e}$ (resp. $Z_e$) that has only $h$ columns (resp. rows). 

\smallskip

The choice of the path $\gamma _{\oB}$ determines a preimage $\omega
_{2}$ of the vector $\p_{2}^{G}$ in $\Z^{E}$ by counting the number of
times (with sign) that $\gamma _{\oB}$ crosses the vanishing cycle
$a_{e}$. Similarly, the form $\omega _{\oA}$ determines a preimage
$\omega _{1}$ of $\p_{1}^{G}$ in $\C^{E}$ with $e$-th component given
by 
\begin{displaymath}
  \int_{a_{e}}\omega _{\oA}.
\end{displaymath}
Recall the definitions of $W_{e}(\omega )$ and $Q_{e}(\omega
_{1},\omega _{2})$ given in Definition \ref{def:1}. 

\begin{prop}\label{prop:3}
  The following equalities hold:
  \begin{displaymath}
    Z_{e} {^{t}\underline{\ps}_{1}}=-W_{e}(\omega _{1}),\qquad
    \underline{\ps}_{2}W_{e} = \,{^{t}}W_{e}(\omega _{2}),
    \qquad 
    \underline{\ps}_{2}\Gamma _{e}
    {^{t}\underline{\ps}_{1}}=-\,Q_{e}(\omega _{1},\omega _{2}). 
  \end{displaymath}
\end{prop}
\begin{proof}
  The
  $j$-th component of $W_{e}(\omega _{2})$ is given by
  \begin{displaymath}
    W_{e}(\omega _{2})_{j}=\langle b_{j},\gamma _{\oB}\rangle_{e}=
    \langle  b_{j},a_e\rangle\langle\gamma _{\oB}, a_e\rangle
    =-c_{e,j}\sum_{l}\ps_{l,2}d_{e,l,2}=(\underline{\ps}_{2}
    W_{e})_{j}. 
  \end{displaymath}
  The
  $i$-th component of $W_{e}(\omega _{1})$ is given by
  \begin{displaymath}
    W_{e}(\omega _{1})_{i}=\langle b_{i},a_e\rangle\int_{a_e}\omega _{\oA}
    =-c_{e,i}\sum_{l} \ps_{l,1}d_{e,l,1}=-(\widetilde Z_{e}
    {^{t}\underline{\ps}_{1}})_{i}. 
  \end{displaymath}
  Finally
  \begin{displaymath}
    Q_{e}(\omega _{1},\omega _{2})=
    \langle\gamma _{\oB}, a_e\rangle \int_{a_e}\omega
    _{\oA}=\sum_{k,l}
    \ps_{k,2}d_{e,k,2}d_{e,l,1}\ps_{l,1}=-\underline{\ps}_{2}\Gamma
    _{e} 
    {^{t}\underline{\ps}_{1}}.
  \end{displaymath}
\end{proof}

\smallskip

Since the monodromy is given by an element of $\widetilde G(\Z)$, the
map $\widetilde{\Phi}$ descends to $U$, making the following diagram
commutative:  
\begin{equation}
\label{diag-phi}
\begin{CD} 
\widetilde{U} @>\widetilde\Phi>> 
\H_{g}\times \text{Row}_{g}(\C)\times \text{Col}_{g}(\C)\times \C \\
@VVV @VVV \\
U @>\Phi>> 
\widetilde{G}(\Z)\backslash\Big(\H_{g}\times \text{Row}_{g}(\C)\times \text{Col}_{g}(\C)\times \C\Big).
\end{CD}
\end{equation}

Clearly the definition of the map $\widetilde \Phi$ can be extended to
the  case when $D=1$ and the divisors are real-valued. But in
this case the monodromy will not be integral valued and the map
$\widetilde \Psi $ will not descend to $U$. Thus there is no
analogue to the diagram 
\eqref{diag-phi}. Finally we extend to the case of $\R^{D}$-valued
divisors simply working componentwise. That is, when the divisors
$\underline \ps_{1}$ and $\underline \ps_{2}$ have values in $\R^{D}$
we define a period map 
\begin{displaymath}
  \widetilde \Psi \colon \widetilde U \longrightarrow 
(\H_{g}\times \text{Row}_{g}(\C)\times \text{Col}_{g}(\C)\times
\C )^{D}. 
\end{displaymath}

\subsection{Asymptotic of the period map  and proof of
  Theorem~\ref{thm:main}}
\label{sec:asympt-peri-map}

The proof of Theorem~\ref{thm:main} is based on the Nilpotent Orbit
Theorem. 
We refer to the paper by Schmidt~\cite{S} and
Cattani-Kaplan-Schmidt~\cite{CKS} for  
the the case of variations of polarized
pure Hodge structures.  
%was proven by Schmid \cite{S}. 
We need the more general case of a variation of mixed Hodge
structures~\cite{KNU, P}.  We actually only need a small part of the
Nilpotent Orbit Theorem that can be found in \cite[Section
6]{pearl:higgs}.

\smallskip

Back to the case of integral valued divisors, consider the
diagram \eqref{diag-phi}. The action of the fundamental group $\Z^E$
of $U$ is unipotent, and we write $N_e$ for the
logarithm of the generator $1_e \in \Z^E$. These operators are given
explicitly in \eqref{eq:14}. To ease notation we write
\begin{displaymath}
  \widetilde X=\H_{g}\times
\mathrm{Row}_{g}(\C)\times \mathrm{Col}_{g}(\C)\times \C.
\end{displaymath}
The untwisted period map
\begin{equation}
  \label{eq:18}
\widetilde\Psi(z)=\exp(-\sum_E z_eN_e)\widetilde\Phi(z)  
\end{equation}
takes values in a ``compact dual'' $\check{\mathcal M}$ which is
(essentially) a flag
variety which parametrizes filtrations $F^*\mathbb C^{g+2}$ which
satisfy the conditions to be the Hodge filtration on a biextension of
genus $g$. The space $\check{\mathcal M}$ contains $\widetilde X$ as
an open subset. It is called the compact dual by analogy with the
theory of semisimple Lie groups although in general is not compact.
Since the map $\widetilde \Psi$ is invariant  under the transformation
$z_e \mapsto z_e+1$, it descends to a map $\Psi: U \to
\check{\mathcal{M}}$.  

\smallskip

 We separate the variables corresponding to the edges as $
 s_E$, and write any point $s$ of $U$ as $s =
 s_E \times s_{E^c}$. The coordinates in the
 universal cover $\widetilde U$ 
 will be denoted by $z_{e}$. The projection $\widetilde U\to U$ is
 given in these coordinates by
 \begin{equation}
   \label{eq:16}
   s_{e}=
   \begin{cases}
     \exp(2\pi i z_{e}),&\text{ for }e\in E,\\
     z_{e},&\text{ for }e\not \in E.
   \end{cases}
 \end{equation}

The following result is the part of the Nilpotent Orbit Theorem that
we need. A proof of it for admissible variations of mixed Hodge
structures can be found in \cite[Section 6]{pearl:higgs}. Recall that
$\Delta $ is a disk of small radius and we denote $S=\Delta ^{3g-3}$.
 
\begin{thm}\label{thm:3}
After shrinking the radius of $\Delta $ if necessary, the map $\Psi$ extends
to a holomorphic map
\begin{displaymath}
  \Psi: S \longrightarrow \check{\mathcal{M}}. 
\end{displaymath}
Moreover, there exists a constant $h_{0}$ such that, if for all $e\in
E$, $\Im(z_{e})\ge h_{0}$, then  
\begin{displaymath}
  \exp(\sum_{e\in E} z_eN_e)\Psi(s)\in \widetilde X.
\end{displaymath}
\end{thm}

\smallskip

We now write $\Psi _0(s) = \exp(\sum_{e} i h_0N_e)\Psi
(s)$ so $\Psi _0\colon S\to \widetilde X$ is holomorphic.
We write
\begin{equation}\label{eq:19}
  \Psi _0(s)=(\Omega _{0}(s),
  W_{0}(s),Z_{0}(s),\rho _{0}(s)),
\end{equation}
\begin{equation}
  \label{eq:15}
  y_{e}=\Im(z_{e})=\frac{-1}{2\pi }\log|s_{e}|.
\end{equation}

\smallskip

Gathering together all the computations we have made we obtain an
expression for the height pairing function.

\begin{prop}\label{prop:2} The height pairing is given by
  \begin{multline}\label{eq:prop57}
    \langle \oA_{s},\oB_{s}\rangle=
    -2\pi \Im(\rho _{0})-\sum_{e\in E}2\pi
    y'_{e}\underline{\ps}_{2}\Gamma 
    _{e} \,{^{t}\underline{\ps}_{1}}+\\
    2\pi \Big(\Im(W_{0})+\sum_{e\in
      E}y'_{e}\underline{\ps}_{2}\widetilde W_{e}\Big)\cdot 
    \Big(\Im(\Omega _{0})+\sum_{e\in E}y'_{e}\widetilde M_{e}\Big)^{-1}\\
   \cdot \Big(\Im(Z_{0})+\sum_{e\in
      E}y'_{e}\widetilde Z_{e}\,{^{t}\underline{\ps}_{1}}\Big), 
  \end{multline}
  where $y_{e}'=y_{e}-h_{0}$.
\end{prop}
\begin{proof}
  By Proposition \ref{height-pm} and equation \eqref{eq:18} we know that
  \begin{displaymath}
    \langle
    \oA_{s},\oB_{s}\rangle=\log\|H_{\oB_{s},\oA_{s}}\| =\log\| \widetilde 
    \Phi (z)\|=\log \|\exp(\sum_{e\in E}z_e\,N_e)\widetilde\Psi(z)\|.
  \end{displaymath}
The proposition follows from the explicit description of the
operators $N_{e}$ in \eqref{eq:14}, and the function $\log \| \cdot \|$ in Theorem
\ref{thm:2}, as well as equations \eqref{eq:19} and \eqref{eq:15}. 
\end{proof}

From the previous proposition we derive the following estimate.

\begin{thm}\label{thm:4}
  After shrinking the radius of $\Delta $ if necessary, the height pairing
  can be written as 
  \begin{multline}
    \label{eq:f}\langle \oA_{s},\oB_{s}\rangle=
    -\sum_{e\in E}2\pi y_{e}\underline{\ps}_{2}\Gamma
    _{e} \,{^{t}\underline{\ps}_{1}}\\
    +2\pi \Big(\sum_{e\in E}y_{e}\underline{\ps}_{2}W_{e}\Big)
    \Big(\sum_{e\in E}y_{e} M_{e}\Big)^{-1}
    \Big(\sum_{e\in
      E}y_{e}Z_{e}\,{^{t}\underline{\ps}_{1}}\Big)+h(s),
  \end{multline}
  where $h\colon U\to \R$ is a bounded function.
\end{thm}

\begin{proof}
  Since $\rho _{0}$ is a holomorphic function on $S=\Delta ^{3g-3}$, after
  shrinking the radius of $\Delta $ we can assume that
  $\Im(\rho _{0})$ is bounded. So we only need to prove that the third
  term in the right hand side of equation \eqref{eq:prop57} is, up to a bounded function, equal to the second term in the right hand side of \eqref{eq:f}. 

\smallskip

  Using that, for any symmetric bilinear form
  $\langle\cdot,\cdot\rangle$ the equality
  \begin{displaymath}
    2\langle a, b\rangle=
    \langle a+b, a+b\rangle-\langle a, a\rangle-\langle b, b\rangle
  \end{displaymath}
  holds, we may assume that $W_{0}={^{t}}Z_{0}$ and that
  $\underline{\ps}_{2}\widetilde
  W_{e}=\underline{\ps}_{1}\,{^{t}}Z_{e}$. 

\smallskip

On the other hand, if we denote by $c_{e}$ (resp. $d_{e,1}$) the column
  vector $(c_{e,i})_{i}$ (resp. $(d_{e,l,1})_{l}$), then
  \begin{displaymath}
    \widetilde M_{e}=c_{e}\, {}^{t}c_{e},\qquad \widetilde 
    Z_{e}=c_{e}\, {}^{t}d_{e}.
  \end{displaymath}
  Therefore we can choose a colum vector $v$ such that
  \begin{displaymath}
    \widetilde Z_{e}\,{^{t}\underline{\ps}_{1}}=\widetilde M_{e}v.
  \end{displaymath}
  For shorthand we write $a=\Im(Z_{0})$ and $B=\Im(\Omega _{0})$.
  Then 
  \begin{displaymath}
    \Big(\,^{t}a+\sum_{e\in
      E}y'_{e}\underline{\ps}_{1}\,{}^{t}\widetilde Z_{e}\Big)\cdot 
    \Big(B+\sum_{e\in E}y'_{e}\widetilde M_{e}\Big)^{-1}\\
   \cdot \Big(a+\sum_{e\in
      E}y'_{e}\widetilde Z_{e}\,{^{t}\underline{\ps}_{1}}\Big),
  \end{displaymath}
 is a normlike function in the terminology of \cite[Section
 3.1]{BdJH}. Taking this into account, the 
  result follows from \cite[Theorem 3.2~(1)]{BdJH}.
  \end{proof}

\begin{remark} A graph-theoretic proof of this theorem based on Equation~\ref{eq:21}, Proposition~\ref{prop:2}, and exchange properties between spanning trees and 2-forests in a graph is given in~\cite{Am}. 
\end{remark}

\smallskip

From Theorem \ref{thm:4}, Proposition \ref{prop:3} and
equation \eqref{eq:21} we derive: 
\begin{cor}\label{cor:1}
  The height pairing
  can be written as 
  \begin{displaymath}
    \langle \oA_{s},\oB_{s}\rangle=
    2\pi \frac{\phi_G(\underline \p_1^G, 
  \underline \p_2^G,\underline y)}{\psi_G(\underline y)}
+h(s),
  \end{displaymath}
  where $\underline y=(y_{e})_{e\in E}$ and $h\colon U\to \R$
  is a bounded function. 
\end{cor}

\begin{remark} \label{rem:1} Let now $\pi \colon \mathcal{C}' \to S'$ denote the
  versal analytic deformation of the marked stable curve $C_0$, and
  let $\sigma_i \colon S \to \mathcal{C}$, for $i=1, \ldots, n$, be
  the sections corresponding to the markings. Then $S$ is a polydisc
  $\Delta^{3g-3+n}$ and the 
  fibres of $\pi$ are smooth over the open subset $U'=(\Delta^\ast)^E
  \times \Delta^{3g-3-|E|+n}$.  

\smallskip
Assume that we are given another family of sections $\lambda_1,
\ldots, \lambda_n$ disjoint between them and from the
$\sigma_i$. Then, for any two 
collections of external momenta $\underline \ps_1=(\ps_{1, i})$ and
$\underline \ps_2=(\ps_{2, j})$ satisfying the conservation law, we
still have 
$$
\langle \sum_{i=1}^n \ps_{1, i} \sigma_i, \sum_{j=1}^n \mathbf{p}_{2, j} \lambda_j \rangle=2\pi \frac{\phi_G(\underline \p_1^G, 
  \underline \p_2^G,\underline y)}{\psi_G(\underline y)}
+h(s)
$$ for a bounded function $h \colon U' \to \R$. 
\end{remark}

\begin{remark}\label{rem:multisec} By considering a ramified covering
  of $S'$, \'etale over $U'$, the same results holds when the
  $\lambda _{i} $ in Remark \ref{rem:1} are multi-valued sections which do not
  meet the double points of $C_{0}$. %such that $\sum_{j=1}^n \mathbf{p}_{2, j} \lambda_j$ is still of degree zero. 
\end{remark}

\begin{proof}[Proof of Theorem \ref{thm:main}] Let $\underline t
  \colon I \to U$ be an admissible segment, and consider a lift
  $\underline{z} \colon I \to \widetilde{U}$. Then
  $t_e(\alpha')=\exp(2\pi i z_e(\alpha'))$, so 
$$
\alpha' y_e(\alpha')=\alpha' \mathrm{Im}(z_e(\alpha'))=-\frac{1}{2\pi} \log |t_e(\alpha')|^{\alpha'}.
$$ In particular, $\lim_{\alpha' \to 0} \alpha' y_e(\alpha)=\frac{1}{2\pi} Y_e$.  

\smallskip

Using Corollary \ref{cor:1}, together with the fact that the quotient of the Symanzik polynomials is homogeneous of degree one, we get
$$
\alpha' \langle \oA_{\underline t(s)},\oB_{\underline t(s)}\rangle=
    2\pi \frac{\phi_G(\underline \p_1^G, 
  \underline \p_2^G, (\alpha' y_e(\alpha'))_{e \in E})}{\psi_G((\alpha' y_e(\alpha'))_{e \in E})}
+\alpha' h(\underline t(s)).
$$ Since the function $h$ is bounded, passing to the limit yields
\begin{equation*}
\pushQED{\qed} 
\lim_{\alpha' \to 0} \alpha' \langle \oA_{\underline t(s)},\oB_{\underline t(s)}\rangle=
\frac{\phi_G(\underline \p_1^G, 
  \underline \p_2^G, \underline Y)}{\psi_G(\underline Y)}.\qquad \qedhere %\popQED
  \end{equation*}
\end{proof}

\section{Convergence of the integrands}
\label{sec:conv-integr}

In this final section, we prove the main result of the paper: the convergence of the integrand in string theory to the integrand of Feynman amplitudes in the low-energy limit $\alpha' \to 0$. For this, we first recall the definition of regularized Green functions. 

\subsection{Green functions} 
\label{sec:green-functions}
Let $C$ be a smooth projective complex curve, together with a smooth
positive $(1, 1)$-form $\mu$. 

\begin{ex} If $C$ has genus $g \geq 1$, a natural choice for $\mu$ is the Arakelov form
$$
\mu_{\Ar}=\frac{i}{2g} \sum_{j=1}^g \omega_i \wedge \overline{\omega_i}, 
$$ where $\omega_1, \ldots, \omega_g$ is any orthonormal basis of the holomorphic differentials $H^0(C, \Omega^1_C)$ for the Hermitian product $(\omega, \omega')=\frac{i}{2} \int_C \omega \wedge \overline{\omega'}$. 
\end{ex}

\smallskip

To $\mu$ one associates a Green function $\green_\mu$ as follows. For a fixed point $x$ of $C$, consider the differential equation
\begin{equation}\label{green-eq}
\partial \overline{\partial} \varphi=\pi i(\delta_x-\mu), 
\end{equation} where $\delta_x$ is the Dirac delta distribution. It
admits a unique solution  
$$
\green_\mu(x, \cdot) \colon C \setminus \{x\} \longrightarrow \R
$$ satisfying the following conditions: 
\begin{itemize}

\item If we choose local coordinates in an analytic chart $U$, then,
  for fixed $x\in U$ there exists a smooth function
  $\alpha$ such that $\green_\mu(x, y)=-\log |y-x|+\alpha(y)$ for any
  $y \in U \setminus \{x\}$.

\smallskip

\item (Normalization) $\int_C \green_\mu(x, y) \mu(y)=0$. 
\end{itemize}

\smallskip

Letting $x$ vary, we can view $\green_\mu$ as a function on $C \times
C \setminus \Delta$. The chosen normalization implies that
$\green_{\mu }$ is symmetric.

\smallskip

The following lemma, proved in \cite[Chap, II, Prop. 1.3]{lang}, explains how the Green function varies when $\mu$ is changed. 

\begin{lem} If $\mu'$ is another positive $(1, 1)$-form on $C$, then there exists a smooth function $f$ on $C$ such that 
\begin{equation}\label{green-dif}
\green_{\mu'}(x, y)=\green_{\mu}(x, y)+f(x)+f(y).
\end{equation}
\end{lem}

The archimedean height pairing between $\R^D$-valued divisors can be expressed in terms of Green functions as follows: 

\begin{lem}\label{lem:h-g} Let $\mathfrak A =\sum \p_{i,1} \sigma_{i,1}$ and $\mathfrak B = \sum \p_{j,2}\sigma_{j,2}$ be $\R^D$-valued degree zero divisors with disjoint support on $C$. Then 
\begin{equation}\label{height-green}
\langle \mathfrak A, \mathfrak B\rangle = \, \sum_{i,j} \langle \p_{i,1}, \p_{j,2} \rangle \green_{\mu}(\sigma_{i,1},\sigma_{j,2})
\end{equation} for any positive $(1, 1)$-form $\mu$ on $C$. In particular, the right hand side of \eqref{height-green} is independent of $\mu$. 
\end{lem}

\begin{proof} By bilinearity, it suffices to prove the result for divisors of the form $\mathfrak A=x_1-x_2$ and $\mathfrak B=y_1-y_2$. For this, consider the function 
$$
\green_{\mathfrak A, \mu}(\cdot)=\green_\mu(x_1, \cdot)-\green_\mu(x_2, \cdot). 
$$ 

\smallskip

We claim that $\omega_{\mathfrak A}=2\partial \green_{\mathfrak A, \mu}$. By Remark \ref{omega-uniq}, this amounts to say that $2\partial \green_{\mathfrak A, \mu}$ has residue $1$ at $x_1$ and $-1$ at $x_2$, and $\int_{\gamma } \partial \green_{\mathfrak A, \mu} \in \R(1)$ for any real-valued cycle $\gamma$ on $C \setminus |\mathfrak A|$. The first property follows from the local expression of $\green_\mu(x_1, \cdot)$ and $\green_\mu(x_2, \cdot)$ around the points $x_1$ and $x_2$, and the second one uses the fact that $\overline{\partial} \green_{\mathfrak A, \mu}=\overline{\partial \green_{\mathfrak A, \mu}}$ since $\green_{\mathfrak A, \mu}$ is a real function. Therefore, 
\begin{align*}
\langle \mathfrak A, \mathfrak B\rangle&=\mathrm{Re} \left(\int_{\gamma_{\mathfrak B}}\omega_{\mathfrak A}\right)=\mathrm{Re} \left(\int_{\gamma_{\mathfrak B}} \partial \green_{\mathfrak A, \mu} +\overline{\partial} \green_{\mathfrak A, \mu} \right)=\mathrm{Re}\left(\int_{\gamma_{\mathfrak B}} d\green_{\mathfrak A, \mu}\right) \\
&=\green_{\mu}(x_1, y_1)-\green_{\mu}(x_1, y_2)-\green_{\mu}(x_2, y_1)+\green_\mu(x_2, y_2),  
\end{align*} as we wanted to show.  
\end{proof}

To prove the convergence of the integrands, we need to extend the definition of the height pairing to divisors with non-disjoint supports. For this we introduce the \textit{regularized Green function} $\green_{\mu}' \colon C \times C \to \R$, which agrees with $\green_\mu$ outside the diagonal, and is defined on $\Delta$ by 
$$
\mathfrak{g}'_{\mu}(x, x)=\lim_{x' \to x} \big(\green_\mu(x', x)+\log d_{\mu }(x',x)\big), 
$$ where $x'$ is a holomorphic coordinate in a small neighborhood of
$x$ and $d_{\mu}$ denotes the distance function associated to the
metric $\mu$.

\smallskip

Replacing the Green function by its regularization in \eqref{height-green}, we can extend the definition of the height pairing to arbitrary $\R^D$-valued divisors and, in particular, define
$$
\langle \mathfrak A, \mathfrak A \rangle'_{\mu}=\sum_{1\le i,j\le n}
\langle \ps_i, \ps_j \rangle \green_{\mu}'(\sigma_i, \sigma_j).
$$

\smallskip

Without further assumptions, the real number $\langle \mathfrak A, \mathfrak A \rangle_\mu'$ depends on the choice of $\mu$. However, we have the following straightforward consequence of Lemma \ref{lem:h-g}: 

\begin{cor} Assume that the external momenta $\ps_i \in \R^D$ satisfy the conservation
  law $\sum_{i=1}^n \ps_i=0$ and the on shell condition
  $\langle \ps_i, \ps_i \rangle=0$ for all $i$. Then $\langle \mathfrak A,
  \mathfrak A \rangle'_\mu$ is independent of the choice of $\mu$.  
\end{cor}
\begin{proof}
  Let $\mu $ and $\mu '$ be two different metrics. Then, using the on
  shell condition, the conservation law and Lemma \ref{lem:h-g}
  \begin{align*}
    \langle \mathfrak A, \mathfrak A \rangle'_\mu&=
    \sum_{1\le i,j\le n} \langle \ps_i, \ps_j \rangle
    \green_{\mu}'(\sigma_i, \sigma_j)
    =
    2\sum_{1\le i<j\le n} \langle \ps_i, \ps_j \rangle
    \green_{\mu}'(\sigma_i, \sigma_j)
    \\&=
    2\sum_{1\le i<j\le n} \langle \ps_i, \ps_j \rangle
    \green_{\mu'}'(\sigma_i, \sigma_j)
    =\langle \mathfrak A, \mathfrak A \rangle'_{\mu'}. \qedhere
  \end{align*}
\end{proof}

\subsection{Asymptotic of the regularized height pairing} Let $\pi \colon \mathcal{C}' \to S'$ be the analytic versal deformation of the stable marked curve $C_0$ over a polydisc $S'=\Delta^{3g-3+n}$, and let $U'=(\Delta^\ast)^{E} \times \Delta^{3g-3-|E|+n} \subset S'$ denote the smooth locus. Consider the $\R^D$-valued relative divisor 
$$
\mathfrak A=\sum_{i=1}^n \ps_i \sigma_i.
$$ 

Given a smooth $(1, 1)$-form $\mu$ on $\pi^{-1}(U')$ such that every restriction $\mu_s=\mu_{| C_s}$ is positive, we get a function $U' \to \R$ by 
$$
\langle \mathfrak A_s, \mathfrak A_s \rangle'_\mu=\sum_{1 \leq i, j \leq n} \langle \ps_i, \ps_j \rangle \green_{\mu_s}'(\sigma_i(s), \sigma_j(s)). 
$$

To study the asymptotic of $\langle \mathfrak A_s, \mathfrak A_s \rangle_\mu'$ as $s$ approaches the boundary, we introduce the following function: 
$$
h_{\ps, \mu}(s)=\langle \mathfrak A_s, \mathfrak A_s \rangle'_\mu-2\pi \frac{\phi_G(\underline \ps^G, \underline y)}{\psi_G(\underline y)}. 
$$

\begin{thm}\label{thm:8} If $\mu$ extends to a continuous $(1, 1)$-form on $\mathcal{C'}$, then the function $h_{\ps, \mu}$ is bounded. 
\end{thm}

\begin{proof} By bilinearity, it suffices to prove that the function is bounded for integer-valued divisors
$$
\mathfrak A=\sum_{i=1}^n \p_i \sigma_i,\quad \p_{i}\in \Z,\quad \sum_{i=1}^n \p_{i}=0. 
$$ 

Let $\Sigma \subset C_0$ denote the union of the set of singular points of $C_0$ and the marked points $\sigma_1(0), \ldots, \sigma_n(0)$. Using the moving lemma, one can find a rational function $f$ on $\mathcal{C}$ such that $\mathfrak A+\mathrm{div}(f)$ does not meet $\Sigma$. Possibly after shrinking $\Delta$, we may assume that the divisor $\mathfrak A +\mathrm{div}(f)$ has support disjoint from $\mathfrak A$ and no vertical components. 

\smallskip

Since $\mathrm{div}(f)$ does not meet the double points of $C_0$ and
$\mathrm{div}(f)_{| X_v}$ has degree zero for each $v \in V$, the
restrictions to $G$ of the momenta of the divisors $\mathfrak A$ and
$\mathfrak A + \mathrm{div}(f)$ coincide. Let $\underline \ps^G$
denote their common value. Then Corollary \ref{cor:1} and remarks
\ref{rem:1} and \ref{rem:multisec} yield
\begin{equation}\label{eq:a+div}
  \langle \mathfrak{A}, \mathfrak{A}+\mathrm{div}(f) \rangle=2\pi
  \frac{\phi_G(\underline \ps^G, \underline y)}{\psi_G(\underline y)}
  + h_1(s),
\end{equation} 
where $\underline y=(y_e)_{e \in E}$ and $h_1 \colon U'\to \R$ is
a bounded function.

\smallskip

Let $\pi_i$ be a local equation of the divisor $\sigma_i\subset \sC'$
around the point 
$\sigma_i(0)$, and consider the first order deformation $\sigma_i^u$
given by $\{\pi_i=u\}$ for $u$ in a small disc. Then the relative
divisor
$$
\mathfrak A^u=\sum_{i=1}^n \p_i \sigma_i^u
$$ 
coincides with $\mathfrak A$ for $u=0$ and is disjoint both from
$\mathfrak A$ and $\mathfrak A+\mathrm{div}(f)$ for $u \neq 0$
sufficiently small. Moreover,
$$
\langle \mathfrak A, \mathfrak A \rangle'_\mu=\lim_{u \to 0} \left( \langle \mathfrak A^u, \mathfrak A \rangle-\sum_{i=1}^n \p_i^2\log d_{\mu }(\sigma_i^u,\sigma_i )  \right). 
$$

\smallskip

By Example \ref{ex:heigh-div}, this can be rewritten as 
\begin{align*}
\langle \mathfrak A, \mathfrak A\rangle_\mu'=\lim_{u \to 0}
  \bigl(\langle \mathfrak A^u, \mathfrak A +\mathrm{div}(f) \rangle
  -\log |f(\mathfrak A^u)|-\sum_{i=1}^n \p_i \log d_{\mu
  }(\sigma_i^u,\sigma_i )\bigr).
\end{align*} 

\smallskip

Note that, since $\mathfrak A + \mathrm{div}(f)$ is disjoint from
$\mathfrak A$, the function $f$ is, locally around the point
$\sigma_i(0)$, of the form $f=\pi_i^{-p_i} v_i$ with $v_i$ invertible,
hence
$$ 
\log |f(\mathfrak A^u)|=-\sum \p_i^2 \log |u| + \p_i\log v_i (\sigma_i^u). 
$$ 
On the other hand, since the metric $\mu$ is continuous, there exists
a continuous function $\eta_\mu$ such that $d_{\mu
}(\sigma_i^u,\sigma_i )=\eta_\mu |u|$. It follows that the function  
$$
h_2(s)=\lim_{u \to 0} \bigl(\log | f(\mathfrak A^u_s)|+\sum_{i=1}^n
\p_i^2 \log d_{\mu }(\sigma_i^u,\sigma_i )\bigr)
$$
is bounded. Combining this with equation \eqref{eq:a+div}, we get
$h_{\ps, \mu}=h_1-h_2$, so it is a bounded function.
\end{proof}

\begin{cor}\label{cor55} Assume that the external momenta satisfy
  $\sum_{i=1}^n \ps_i=0$ and $\langle \ps_i, \ps_i \rangle=0$ for all
  $i$. Then the function $\green_{\ps, \mu}$ is independent of
  $\mu$. In particular, when $g \geq 1$ and $\mu=\mu_{\Ar}$, the
  following holds
\begin{equation}\label{eq: bound-mu}
\langle \mathfrak A, \mathfrak A \rangle'_{\mu_{Ar}}=2\pi \frac{\phi_G(\ps^G,
  \underline y)}{\psi_G(\underline y)}+\text{bounded}.
\end{equation}
\end{cor}

\begin{remark} \rm By \cite[Thm 1.2]{dJ}, when the base is one
  dimensional, the Arakelov metric has logarithmic singularities. This
  implies that equation \eqref{eq: bound-mu} does not hold for $\mu_{\Ar}$ without the ``on shell'' condition $\langle
  \ps_i, \ps_i \rangle=0$. Since the asymptotic behaviour of the
  Arakelov metric is also determined by the combinatorics of the dual
  graph of $C_{0}$, one may ask what the asymptotic of $\langle \mathfrak
  A, \mathfrak A \rangle_{\mu_{Ar}}$ is in the general case. We have a formula in terms of the Green's function associated to the Zhang measure on the metric graph $G$ with edge lengths $\underline{Y}$. We hope to return to this point in a future publication. 
\end{remark}

From Corollary \ref{cor55} we immediately derive: 

\begin{thm}\label{thm:5} Assume that the external momenta satisfy
  $\sum_{i=1}^n \ps_i=0$ and $\langle \ps_i, \ps_i \rangle=0$ for all
  $i$. Then, for any admissible segment $\underline t \colon I \to U'$,
  the following holds:
  $$
  \lim_{\alpha' \to 0} \alpha' \langle \mathfrak A_{\underline t(\alpha')},
  \mathfrak A_{\underline t(\alpha')}\rangle'_{\mu_{Ar}}=\frac{\phi_G(\ps^G,
    \underline Y)}{\psi_G(\underline Y)}.
  $$
\end{thm}

\end{document}